\documentclass{amsart}
 
 \usepackage{amsmath}
\usepackage{amssymb}
\usepackage{amsfonts}

\usepackage{alltt}
\usepackage{amsthm}
\usepackage{bm}
\usepackage{mathrsfs}
\usepackage{graphicx}
\usepackage{color}
\usepackage{mathtools}
\usepackage[T1]{fontenc}
\usepackage[utf8]{inputenc}

\usepackage{graphicx} 
	\usepackage{tikz,pgfplots}

\newcommand{\id}{\operatorname{Id}}

\newcommand{\C}{\mathbb C}
\newcommand{\R}{\mathbb R}
\newcommand{\N}{\mathbb N}
\newcommand{\Z}{\mathbb Z}
\newcommand{\ES}{\mathbb S}

\newcommand{\de}{\, d}

\newcommand{\re}{\operatorname{Re}}
\newcommand{\im}{\operatorname{Im}}
\newcommand{\ve}{\varepsilon}
\newcommand{\hdot}{\bm{\ldotp}}

\newcommand{\spec}{\operatorname{Spec}}
\newcommand{\Arg}{\operatorname{Arg}}

\newcommand{\tr}{\operatorname{tr}}

\numberwithin{equation}{section}

\newtheorem{theorem}{Theorem}[section]
\newtheorem{proposition}[theorem]{Proposition}
\newtheorem{lemma}[theorem]{Lemma}

\newtheorem*{thm*}{Theorem}
\newtheorem*{lemma*}{Lemma}
\newtheorem*{prop*}{Proposition}

\theoremstyle{definition}
\newtheorem{definition}[theorem]{Definition}
\newtheorem{example}[theorem]{Example}

\newtheorem*{dfn*}{Definition}

\theoremstyle{remark}
\newtheorem{remark}[theorem]{Remark}

\makeatletter
\let\save@mathaccent\mathaccent
\newcommand*\if@single[3]{%
  \setbox0\hbox{${\mathaccent"0362{#1}}^H$}%
  \setbox2\hbox{${\mathaccent"0362{\kern0pt#1}}^H$}%
  \ifdim\ht0=\ht2 #3\else #2\fi
  }
\newcommand*\rel@kern[1]{\kern#1\dimexpr\macc@kerna}
\newcommand*\widebar[1]{\@ifnextchar^{{\wide@bar{#1}{0}}}{\wide@bar{#1}{1}}}
\newcommand*\wide@bar[2]{\if@single{#1}{\wide@bar@{#1}{#2}{1}}{\wide@bar@{#1}{#2}{2}}}
\newcommand*\wide@bar@[3]{%
  \begingroup
  \def\mathaccent##1##2{%
    \let\mathaccent\save@mathaccent
    \if#32 \let\macc@nucleus\first@char \fi
    \setbox\z@\hbox{$\macc@style{\macc@nucleus}_{}$}%
    \setbox\tw@\hbox{$\macc@style{\macc@nucleus}{}_{}$}%
    \dimen@\wd\tw@
    \advance\dimen@-\wd\z@
    \divide\dimen@ 3
    \@tempdima\wd\tw@
    \advance\@tempdima-\scriptspace
    \divide\@tempdima 10
    \advance\dimen@-\@tempdima
    \ifdim\dimen@>\z@ \dimen@0pt\fi
    \rel@kern{0.6}\kern-\dimen@
    \if#31
      \overline{\rel@kern{-0.6}\kern\dimen@\macc@nucleus\rel@kern{0.4}\kern\dimen@}%
      \advance\dimen@0.4\dimexpr\macc@kerna
      \let\final@kern#2%
      \ifdim\dimen@<\z@ \let\final@kern1\fi
      \if\final@kern1 \kern-\dimen@\fi
    \else
      \overline{\rel@kern{-0.6}\kern\dimen@#1}%
    \fi
  }%
  \macc@depth\@ne
  \let\math@bgroup\@empty \let\math@egroup\macc@set@skewchar
  \mathsurround\z@ \frozen@everymath{\mathgroup\macc@group\relax}%
  \macc@set@skewchar\relax
  \let\mathaccentV\macc@nested@a
  \if#31
    \macc@nested@a\relax111{#1}%
  \else
    \def\gobble@till@marker##1\endmarker{}%
    \futurelet\first@char\gobble@till@marker#1\endmarker
    \ifcat\noexpand\first@char A\else
      \def\first@char{}%
    \fi
    \macc@nested@a\relax111{\first@char}%
  \fi
  \endgroup
}
\makeatother

\makeatletter
\def\@setauthors{%
  \begingroup
  \def\thanks{\protect\thanks@warning}%
  \trivlist
  \centering\large \@topsep30\p@\relax
  \advance\@topsep by -\baselineskip
  \item\relax
  \author@andify\authors
  \def\\{\protect\linebreak}%
  \authors%
  \ifx\@empty\contribs
  \else
    ,\penalty-3 \space \@setcontribs
    \@closetoccontribs
  \fi
  \endtrivlist
  \endgroup
}
\def\@settitle{\begin{center}%
  \baselineskip14\p@\relax
    \normalfont\LARGE
  \@title
  \end{center}%
}
\makeatother

\pgfplotsset{compat=1.13}

\begin{document}

\title[Periodic Zakharov-Shabat systems]{Quantization conditions of eigenvalues for semiclassical \\ Zakharov-Shabat systems on the circle}

\date{\today}

\author{Setsuro Fujii\'e}
\address[Setsuro Fujiie]{Department of Mathematical Sciences, Ritsumeikan University, Ku\-sa\-tsu, 525-8577, Japan}
\email{fujiie@fc.ritsumei.ac.jp}

\author{Jens Wittsten}
\address[Jens Wittsten]{Department of Mathematical Sciences, Ritsumeikan University, Ku\-sa\-tsu, 525-8577, Japan}
\curraddr[Jens Wittsten]{Centre for Mathematical Sciences, Lund University, Box 118, 221 00, Lund, Sweden}
\email{jensw@maths.lth.se}



\begin{abstract}
Bohr-Sommerfeld type quantization conditions of semiclassical eigenvalues for the non-selfadjoint Zakharov-Shabat operator on the unit circle are derived using an exact WKB method. The conditions are given in terms of the action associated with the unit circle or the action associated with turning points following the absence or presence of real turning points.
\end{abstract}

\maketitle

\section{Introduction}

We consider the eigenvalue problem
\begin{equation}
\label{evp}
P(h)u=\lambda u
\end{equation}
 for the Zakharov-Shabat operator
\begin{equation*}
P(h)=\bigg( \begin{array}{cc} -hD_x & iV(x)\\ iV(x) & hD_x\end{array} \bigg),\quad D_x=-i\partial/\partial x,
\end{equation*}
where $u$ is a column vector, $h$ is a small positive parameter, $\lambda$ is a spectral parameter, and $V$ is a real valued analytic function on $\ES^1=\R/2\pi\Z$.
The eigenvalue problem \eqref{evp} appears in the inverse scattering method for the initial value problem for the focusing nonlinear Schr\"odinger equation as one half of the Lax pair \cite{shabat1972exact}.
It will also be written in the form
\begin{equation}\label{system}
\frac{h}{i}\frac{du}{dx}=M(x,\lambda)u,\quad M(x,\lambda)=\bigg( \begin{array}{cc} -\lambda & iV(x)\\ -iV(x) & \lambda\end{array} \bigg).
\end{equation}

The operator $P(h)$ is not selfadjoint. To study  the spectrum $\spec P(h)$, let
\begin{equation}\label{symbolP}
P(x,\xi)=\bigg( \begin{array}{cc} -\xi & iV(x)\\ iV(x) & \xi\end{array} \bigg)
\end{equation}
be the  semiclassical symbol of the operator $P(h)$. 
We define the closure of the set of eigenvalues of $P(x,\xi)$ by 
\begin{equation*}
\Sigma(P)=\overline{\{\lambda\in \C:\exists (x,\xi)\in T^\ast \ES^1,\ \det(P(x,\xi)-\lambda\id)=0\}},
\end{equation*}
where $\det(A)$ denotes the determinant of the matrix $A$.  
Thus, in our case
\begin{equation*}
\Sigma(P)
=\R\bigcup i[-V_0,V_0],
\end{equation*}
where $V_0:=\max_{x\in \ES^1} \lvert V(x)\rvert$.
By Proposition \ref{prop:spectrumrestriction} below, the spectrum of $P(h)$ is discrete and concentrates on $\Sigma(P)$ as $h\to0$.
Hence, to study its asymptotic form one can assume that the spectral parameter $\lambda$ belongs to a small neighborhood of $\R\bigcup i[-V_0,V_0]$.

Before stating our main results, we recall that the roots of $\det(M(x,\lambda))=0$, or equivalently,
\begin{equation*}
V(x)^2+\lambda^2=0,
\end{equation*}
are called {\it turning points} of the system \eqref{system}. A zero of order $n$ of $V^2+\lambda^2$ is called a turning point of order $n$. When $n=1$ or $n=2$, the turning point is said to be simple or double, respectively. When $\lambda=0$, the simple zeros of $V$ are double turning points. Double (or higher order) turning points also occur when $\lambda=i\mu$ and $\pm\mu$ is a local extreme value of $V$.  There are no real turning points for non-zero real $\lambda$, and if $V_1:=\min_{x\in \ES^1} |V(x)|$ is strictly positive, then there are no real turning points for $\lambda=i\mu$ with $-V_1< \mu<V_1$.

To describe the asymptotic distribution of eigenvalues in the semiclassical limit, 
we fix $\lambda_0\in \R\bigcup i[-V_0,V_0]$ and study the quantization condition of eigenvalues in a small complex neighborhood of $\lambda_0$.
The form of the quantization condition depends on whether there are real turning points for $\lambda_0$.
We begin with the turning point free case, which is considerably easier. Let $B_\ve(\lambda_0)$ denote the disc of radius $\ve>0$ centered at $\lambda_0$.
We then have the following Bohr-Sommerfeld type quantization condition
associated with the action integral along $\ES^1$ given by
\begin{equation*}
I(\lambda)=\int_0^{2\pi}(V(t)^2+\lambda^2)^{1/2}\de t.
\end{equation*}

\begin{theorem}\label{thm:A}
Let $\lambda_0\in\R\setminus \{0\}$ if $V_1=0$ or $\lambda_0\in\R\bigcup i(-V_1,V_1)$ if $V_1>0$.
Then there exists an $\varepsilon>0$ and a function $r_1(\lambda,h)$, 
analytic with respect to $\lambda$ in $B_\ve(\lambda_0)$ and uniformly of $O(h)$ as $h\to0$, such that $\lambda\in B_\ve(\lambda_0)$ is an eigenvalue of $P(h)$ if and only if 
\begin{equation}\label{eq:theoremA}
2\cos (I/h)+r_1(\lambda,h)=2.
\end{equation}
In particular, for any small $h$ there is an integer $k\in\Z$ such that
\begin{equation}\label{quantizationcondition1}
I(\lambda)=2\pi kh+O(h^2).
\end{equation}
\end{theorem}

Next we consider $\lambda_0=i\mu_0$, $\mu_0\in (-V_0,-V_1)\bigcup  (V_1,V_0)$, for which there is at least one real turning point. 
In view of the symmetry of the eigenvalues with respect to the real axis (Lemma \ref{symmetrylemma}), it is enough to consider $\mu_0\in (V_1,V_0)$. 
In this paper, we only consider those $\lambda_0\in i(V_1,V_0)$ for which the turning points are all simple.
This implies that the turning points for $\lambda=i\mu$ with $\mu$ close enough to $\mu_0$ stay simple and depend analytically on $\mu$ (in what follows, $\mu$ is reserved for the complex number defined via $\lambda=i\mu$).
For each such $\mu$, there are an even number of turning points $\{x_j(\mu)\}_{j=1}^{2l}$ which, for real $\mu$, are real  and ordered by $0\le  x_1(\mu)<x_2(\mu)<\ldots <x_{2l}(\mu)<2\pi$.
Assuming without loss of generality that $V(0)=V_0$, we have that $x_1(\mu_0)$ is positive and $V(x)^2> \mu_0^2$ for $x_{2j}(\mu_0)<x<x_{2j+1}(\mu_0)$,
$1\le j\le l$, with the convention $x_{2l+1}=x_1$.
We define two kinds of action integrals $S_j(\mu)$ and $I_j(\mu)$ for $1\le j\le l$,
which are real for real $\mu$ close to $\mu_0$, by
\begin{equation}\label{actionS}
S_j(\mu)=\int_{x_{2j-1}(\mu)}^{x_{2j}(\mu)}(\mu^2-V(t)^2)^{1/2}\de t,
 \quad  
I_j(\mu)=\int_{x_{2j}(\mu)}^{x_{2j+1}(\mu)}(V(t)^2-\mu^2)^{1/2}\de t.
\end{equation}

\begin{theorem}\label{thm:B}
Suppose $\lambda_0=i\mu_0$, $\mu_0\in  (V_1,V_0)$, and that the turning points are all simple.
Then there exists an $\varepsilon>0$ and a function $r_2(\lambda,h)$, 
analytic with respect to $\lambda$ in $B_\ve(\lambda_0)$ and uniformly of $O(h)$ as $h\to0$, such that $\lambda\in B_\ve(\lambda_0)$ is an eigenvalue of $P(h)$ if and only if 
\begin{equation}\label{eq:theoremB}
\exp\bigg(\sum_{j=1}^l S_j/h\bigg)\bigg(2^l\prod_{j=1}^l\cos(I_j/h)+r_2(\lambda,h)\bigg)=2.
\end{equation}
In particular, for any small $h$ there are integers $1\le j\le l$ and $k\in\Z$ such that
\begin{equation}\label{quantizationcondition2}
I_j(\mu)=\big(\tfrac{1}{2}+k\big)\pi h+O\big(h^{1+1/l}\big).
\end{equation}
\end{theorem}

This work was first inspired by the paper by Galtsev and Shafarevich \cite{Galtsev2006} who treated the Schr\"odinger operator with complex potential $\mathfrak{D}=-h^2d^2/dx^2+i\cos x$ on $\ES^1$. They showed that the spectrum of this non-selfadjoint operator concentrates on a rotated `Y' shape in the semiclassical limit $h\to0$ while the set $\Sigma(\mathfrak{D})$
is the half band $\{\lambda\in \C: \re \lambda\ge 0,\ -1\le \im\lambda\le 1\}$.
 This fact has also been used by Dyatlov and Zworski \cite{dyatlov2015stochastic} to provide a negative example of stochastic stability of resonances in the context of Anosov flows.
The numerical range $\Sigma(P)$ of our operator $P(h)$, on the contrary, is included in the real and imaginary axes and the spectrum $\spec P(h)$ concentrates
on the whole numerical range. Thus, $P(h)$ does not share the behavior of $\mathfrak{D}$ whose spectrum concentrates on a thin subset of $\Sigma(\mathfrak{D})$. On the other hand, the fact that $\Sigma(P)$ is just lines actually specifies the Stokes geometry near the real axis, which allows us to treat general potentials.

To obtain the quantization conditions, we use the exact WKB method along the lines of \cite{fujiie2009semiclassical}, first introduced by Ecalle \cite{ecalle1984cinq} and used by G\'erard and Grigis \cite{gerard1988precise} to study the Schr\"odinger operator. An exact WKB solution is a convergent resummation of the WKB asymptotic expansion in a turning-point-free complex region,
and the connections of such solutions in different regions via the Wronskian formula (see \S2.2) enable us to get the global asymptotic behavior of solutions which leads to the quantization condition of eigenvalues.
The asymptotic property of the exact WKB  solution is valid only away from turning points, and this prevents us from computing the Wronskian between two exact WKB solutions
whose common region of validity is pinched by two turning points close to each other
(of distance $O(h)$). That is why we exclude energies near $\lambda_0$ for which $V$ has double or higher order turning points.

The main difference compared with the Schr\"odinger case is that for our Zakharov-Shabat operator there are two types of turning points, ones of which are the zeros of $V(x)-\mu$ and the others
of which are the zeros of $V(x)+\mu$. 
The phase function of the WKB solutions is a primitive function of the square root of their product $V(x)^2-\mu^2$ while the amplitude  is a function of  the fourth root of their quotient (see \eqref{defQ}).
This means that the same Stokes geometry, determined by the phase function, may produce different so-called transition matrices (describing the connections between WKB solutions in different domains) according to the type of intermediate turning points.
However, it turns out that the trace of the transition matrix is essentially unaffected (see Proposition \ref{prop:trace}) which explains why there is no mention in Theorem \ref{thm:B} of the type of turning points involved. This is not the case in the recent work by Hirota \cite{HirotaZS} about the eigenvalue problem for a semiclassical Zakharov-Shabat operator (corresponding to the defocusing nonlinear Schr\"odinger equation) on the real line. In \cite{HirotaZS}, two cases are studied: a simple well potential (two turning points of the same type), and a monotone potential (two turning points of opposite type), resulting in different quantization conditions for the corresponding Zakharov-Shabat system. In our study, there are necessarily an even number of turning points by periodicity, which corresponds (in spirit) to the case of a simple well potential in \cite{HirotaZS}.

Here we also mention the study of Gr{\'e}bert and Kappeler \cite{grebert2001estimates} of the periodic eigenvalues of a Zakharov-Shabat operator in the high-energy regime. The problem in the high-energy limit is equivalent to that in the semiclassical limit with a fixed positive energy and with a potential of order $h$. This small potential can be regarded as a small perturbation which does not affect the principal asymptotics of the eigenvalue distribution, and a slight modification of Theorem 1.1 would give \eqref{quantizationcondition1} with $I(\lambda)=2\pi\lambda$, the action for $V=0$.

The study of non-selfadjoint Zakharov-Shabat systems on the real line has a long history in connection with inverse scattering theory. For the study in the semiclassical limit, we refer to the book by Kamvissis, McLaughlin and Miller \cite{kamvissis2003semiclassical}. Real energies belong to the continuous spectrum and the reflection coefficient is relevant. Energies near the imaginary axis consist of eigenvalues when $V$ decays at infinity, and has a band structure when $V$ is periodic (see for example Korotyaev and Kargaev \cite{korotyaev2010estimates}).

The paper is organized as follows.
The exact WKB method is reviewed in Section \ref{section:preliminaries}, while Section \ref{section:proofA} contains the proof of Theorem \ref{thm:A}. 
The proof of Theorem \ref{thm:B} is the content of Section \ref{section:proofB}. A crucial detail is the computation of structure formulas for the transition matrices, which for Theorem \ref{thm:B} is a much more involved and delicate affair. These computations can be found in Section \ref{section:sheets}.

\section{Preliminaries}\label{section:preliminaries}

We identify $\ES^1$ with the fundamental domain $[0,2\pi)\subset\R$, and $V$ with a $2\pi$-periodic function on $\R$.
In this context we regard the symbol $P(x,\xi)$ in \eqref{symbolP} as a function on $T^\ast\R$ which is periodic in $x$,
and $P(h)$ as an operator acting on vector-valued periodic functions on $\R$ belonging to $L^2([0,2\pi),\C^2)$.
As mentioned in the introduction, 
for small $h>0$ it suffices to study the case when $\lambda$ belongs to a small neighborhood of the set of eigenvalues of the symbol $P(x,\xi)$.

\begin{proposition}\label{prop:spectrumrestriction}
For $z$ in the resolvent set, $(P(h)-z\id)^{-1}$ is a compact operator and $P(h)$ has discrete spectrum.
Moreover, if $\Omega$ is an open connected set such that 
$\overline{\Omega}\bigcap\Sigma(P)=\emptyset$, then $(P(h)-z\id)^{-1}$ is a holomorphic function of $z\in\Omega$ provided that $h$ is sufficiently small.
\end{proposition}

\begin{proof}
By \cite[Theorem 6.29]{kato1966perturbation} the first statement follows if we show that $(P(h)-z\id)^{-1}$ is a compact operator for some $z$ in the resolvent set. For such $z$ we have $(P(h)-z\id)^{-1}\in\Psi^{-1}$ by the calculus, where $\Psi^k$ denotes the semiclassical pseudodifferential operators of order $k$. In view of the theorem of Rellich and Kondrachov (see e.g., \cite[Section 6.3]{Adams}) this implies that $(P(h)-z\id)^{-1}$ is compact on $L^2([0,2\pi),\C^2)$.

To prove the second statement we shall adapt the arguments of Dencker \cite{pseudosys} to our situation.
Note that the symbol $P(x,\xi)$ is also the principal symbol of $P(h)$, and that when we regard $P(x,\xi)$ as a function on $T^\ast\R$ the interior of $\Sigma(P)$ consists of those $\lambda$ for which $\det(P(x,\xi)-\lambda\id)=0$ for some $(x,\xi)\in T^\ast\R$. Similarly, we introduce the eigenvalues at infinity,
\begin{multline*}
\Sigma_\infty(P)=\{\lambda\in \C:\exists (x_j,\xi_j)\to\infty,\ \exists u_j\in\C^2\setminus 0 \text{ such that }\\ |P(x_j,\xi_j)u_j-\lambda u_j|/|u_j|\to 0\text{ as } j\to\infty\},
\end{multline*}
which is easily seen to be closed in $\C$ by taking a suitable diagonal sequence.
We now claim that $\Sigma(P)=\Sigma_\infty(P)$ in our case. Indeed,
$\Sigma(P)\subset\Sigma_\infty(P)$ by definition.
Conversely, let $\lambda\in\Sigma_\infty(P)$, then $|P(x_j,\xi_j)u_j-\lambda u_j|/|u_j|\to0$ as $j\to\infty$ for some $(x_j,\xi_j)\in T^\ast\R$ and $0\ne u_j\in\C^2$. We cannot have  $|\xi_j|\to\infty$ since this would imply that $|P(x_j,\xi_j)u_j-\lambda u_j|/|u_j|>1$ if $j$ is sufficiently large. Since $\{u_j/|u_j|\}_j$ and $\{V(x_j)\}_j$ are also bounded, we find by restricting to a subsequence and passing to the limit as $j\to\infty$ that
\begin{equation*}
\bigg(\begin{array}{cc} -\xi-\lambda& ic\\ ic& \xi-\lambda \end{array}\bigg)u=0,
\end{equation*}
where $|u|=1$ and $c$ belongs to the set of values of $V$. But then $\lambda\in\Sigma(P)$ which proves the claim.

Next, using the weight $m(x,\xi)=(1+|\xi|^2)^{1/2}$ we introduce the symbol classes $S(m^k)$, $k\in \Z$, of matrix valued $p\in C^\infty(T^\ast\R,\mathcal L(\C^2,\C^2))$ such that
\begin{equation*}
\|\partial_x^\alpha\partial_\xi^\beta p(x,\xi)\|\le C_{\alpha\beta}m(x,\xi)^k\quad\forall\, \alpha,\beta
\end{equation*}
where $\|A\|$ denotes the norm of the matrix $A$. Then $P\in S(m)$.
Using the Frobenius norm which satisfies $\|\phantom{i}\|\le\|\phantom{i}\|_F$, one easily checks that $\|P(x,\xi)^{-1}\|\le C(1+|\xi|)^{-1}$ when $|\xi|\gg1$.
If $z_1\notin\Sigma(P)$ we thus find by \cite[Proposition 2.20]{pseudosys} that $P(h)-z_1\id$ is invertible\footnote{Actually, \cite[Proposition 2.20]{pseudosys} concerns systems on $\R^n$ but the claim follows by inspection of the proof together with the fact that symbols in $S(1)$ yield bounded operators on $L^2([0,2\pi),\C^n)$ by (an extension of) \cite[Theorem 5.5]{Zworski}.} for sufficiently small $h$, so we can define 
\begin{equation*}
Q(h)=(P(h)-z_1\id)^{-1}(P(h)-z_2\id),\quad z_1\neq z_2,
\end{equation*}
and by the calculus, the symbol of $Q(h)$ is in $S(1)$, i.e., the symbol and all its derivatives are bounded. We take $z_1\notin\Omega$ which is possible by assumption. Now,
\begin{equation*}
\Sigma(Q)=\{\zeta\in\C:(\zeta z_1-z_2)/(\zeta-1)\in\Sigma(P)\},
\end{equation*}
so $\zeta\in\Sigma(Q)$ if and only if $\zeta=(z-z_2)/(z-z_1)$ for some $z\in\Sigma(P)$. Note that $\zeta\ne1$ since $z_1\ne z_2$. 
If $\Omega_1=\{(z-z_2)/(z-z_1):z\in\Omega\}$, then $\Omega_1$ is an open connected set such that 
$\overline{\Omega}_1\bigcap\Sigma(Q)=\emptyset$.
By \cite[Proposition 2.19]{pseudosys} it follows that $(Q(h)-\zeta\id)^{-1}$ is holomorphic in $\zeta\in\Omega_1$ for $h$ sufficiently small.
Since the resolvents of $P(h)$ and $Q(h)$ are related by
\begin{equation*}
(Q(h)-\zeta\id)^{-1}=(1-\zeta)^{-1}(P(h)-z_1\id)\bigg(P(h)-\frac{\zeta z_1-z_2}{\zeta-1}\id\bigg)^{-1}
\end{equation*}
the result follows.
\end{proof}

\subsection{Exact WKB solutions}

Here we recall the construction of a solution of \eqref{system} in a complex domain as a convergent series.
Since $V$ is real analytic and periodic, it follows that there is a $\delta>0$ such that $V$ extends to a holomorphic function in the strip
\begin{equation}\label{strip}
D=\{ x\in\C : \lvert\im x\rvert<\delta\}.
\end{equation}
The exact WKB solutions of systems of type \eqref{system} are known to be of the form
\begin{equation}\label{exactWKB}
u^\pm(x,h)=e^{\pm z(x)/h}\bigg( \begin{array}{cc} 1 & 1\\ -1 & 1\end{array} \bigg)Q(z(x))\bigg( \begin{array}{cc} 0 & 1\\ 1 & 0\end{array} \bigg)^{(1\pm1)/2}w^\pm(x,h),
\end{equation}
see \cite{fujiie2009semiclassical}. The function $z(x)$ is the complex change of coordinates
\begin{equation}\label{z}
z(x)=z(x;x_0)=i\int_{x_0}^x\sqrt{V(t)^2+\lambda^2}\de t,
\end{equation}
where $x_0$ is a base point in $D$. Here, $z(x)$ is defined on the Riemann surface of $(V^2+\lambda^2)^{1/2}$ over $D$, and $Q$ is the matrix valued function
\begin{equation}\label{defQ}
Q(z)=\bigg( \begin{array}{cc} H(z)^{-1} & H(z)^{-1}\\ iH(z) & -iH(z)\end{array} \bigg)
\quad \text{with }
H(z(x))=\bigg(\frac{iV(x)+\lambda}{iV(x)-\lambda}\bigg)^{1/4}
\end{equation}
defined on the Riemann surface of 
$H(z({\hdot}))$ over $D$.  
These Riemann surfaces are defined by introducing branch cuts emanating from the zeros of $x\mapsto\det(M(x,\lambda))$, i.e., of $iV\pm\lambda$ (the turning points of the system \eqref{system}), see Section \ref{section:proofB}.

The amplitude vectors $w^\pm$ in \eqref{exactWKB} are defined as the (formal) series
\begin{equation}\label{amplitude}
w^\pm(x,h)=\bigg( \begin{array}{c} w^\pm_{\mathrm{even}}(x,h)\\ w^\pm_{\mathrm{odd}}(x,h)\end{array} \bigg)
=\sum_{n=0}^\infty 
\bigg( \begin{array}{c} w^\pm_{2n}(z(x))\\ w^\pm_{2n+1}(z(x))\end{array} \bigg),
\end{equation}
where $w^\pm_0(z)\equiv 1$, while $w^\pm_j(z)$ for $j\ge1$ are the unique solutions to the scalar transport equations
\begin{equation}\label{transportodd}
\bigg(\frac{d}{dz}\pm\frac{2}{h}\bigg)w^\pm_{2n+1}(z)=\frac{dH(z)/dz}{H(z)}w^\pm_{2n}(z),
\end{equation}
\begin{equation}\label{transporteven}
\frac{d}{dz}w^\pm_{2n+2}(z)=\frac{dH(z)/dz}{H(z)}w^\pm_{2n+1}(z)
\end{equation}
with prescribed initial conditions $w_n^\pm(\tilde z)=0$ for some choice of base point $\tilde z=z(\tilde x)$ where $\tilde x$ is not a turning point. Here $d/dz$ is defined through the chain rule, e.g.,
\begin{equation}\label{d/dz}
\frac{d}{dx}H(z(x))=\frac{dH(z)}{dz}z'(x).
\end{equation}
Note that these equations are the same as those obtained by an exact WKB construction for scalar Schr{\"o}dinger equations \cite{gerard1988precise,ramond1996semiclassical}.
When we want to signify the dependence on the base point $\tilde z=z(\tilde x)$ we write
\begin{equation*}
w^\pm(x,h;\tilde x)
=\bigg( \begin{array}{c} w^\pm_{\mathrm{even}}(x,h;\tilde x)\\ w^\pm_{\mathrm{odd}}(x,h;\tilde x)\end{array} \bigg)
\end{equation*}
for the amplitude vectors.

Let $\Omega$ be a simply connected open subset of $D$, free from turning points. Then $z=z(x)$ is conformal from $\Omega$ onto $z(\Omega)$. For fixed $h>0$, the formal series \eqref{amplitude} converges uniformly in a neighborhood of the amplitude base point $\tilde x$, and $w^\pm_{\mathrm{even}}(x,h)$ and $w^\pm_{\mathrm{odd}}(x,h)$ are analytic functions in $\Omega$, see \cite[Lemma 3.2]{fujiie2009semiclassical}. It follows that the functions $u^\pm$ given by \eqref{exactWKB} are exact solutions of \eqref{system}.
They shall henceforth be written as
\begin{equation}\label{u}
u^\pm(x;x_0,\tilde x)
\end{equation}
to indicate the particular choice of amplitude base point $\tilde x\in \Omega$, and phase base point $x_0\in D$ as it appears in \eqref{z}.
Note that these solutions are defined for example everywhere on $\R$, although some of the expressions involved are only defined on Riemann surfaces of $(V^2+\lambda^2)^{1/2}$ or $H(z({\hdot}))$.

For fixed $\tilde x\in\Omega$, let $\Omega_\pm$ be the set of points $x$ for which there is a path $\Gamma(\tilde x,x)$ from $\tilde x$ to $x$ along which $t\mapsto \pm\re z(t)$ is strictly increasing. In other words, $x\in\Omega_\pm$ if there is a path $\Gamma(\tilde x,x)$ which intersects the {\it Stokes lines} (i.e., the level curves of $t\mapsto \re z(t)$) transversally in the appropriate direction. The calculation of the quantization condition will rely on the following asymptotic properties.

\begin{remark}\label{h-asymptotic}
For any $k,N\in\N$
\begin{equation*}
\partial^k\big(w^\pm_\mathrm{even}(x,h)-\sum_0^Nw_{2n}^\pm(z(x))\big)=O(h^{N+1}),
\end{equation*}
\begin{equation*}
\partial^k\big(w^\pm_\mathrm{odd}(x,h)-\sum_0^Nw_{2n+1}^\pm(z(x))\big)=O(h^{N+2}),
\end{equation*}
uniformly on compact subsets of $\Omega_\pm$ as $h\to0$, see \cite[Proposition 3.3]{fujiie2009semiclassical}.
In particular,
\begin{equation*}
w^\pm_\mathrm{even}(x,h)=1+O(h),\quad w^\pm_\mathrm{odd}(x,h)=O(h),
\end{equation*}
as $h\to0$.
\end{remark}

\subsection{The Wronskian formula}

For vector valued solutions $u$ and $v$ of \eqref{system}, we introduce the {\it Wronskian}
\begin{equation*}
\mathcal W(u,v)(x)=\det(u(x)\ v(x)).
\end{equation*}
Since the trace of the matrix $M(x,\lambda)$ is zero, $\mathcal W(u,v)$ is actually independent of $x$. 
For a phase base point $x_0\in D$ and different amplitude base points $\tilde x,\tilde y\in\Omega$, elementary computations show that
\begin{multline*}
\mathcal W(u^+(x;x_0,\tilde x),u^-(x;x_0,\tilde y))
\\=-4i\big(
w_\mathrm{odd}^+(x,h;\tilde x)w_\mathrm{odd}^-(x,h;\tilde y)-
w_\mathrm{even}^+(x,h;\tilde x)w_\mathrm{even}^-(x,h;\tilde y)\big),
\end{multline*}
where $u^\pm$ are given by \eqref{exactWKB} via \eqref{u}, and we used the fact that $\det(Q)=-2i$. Since the Wronskian is independent of $x$, we can choose $x=\tilde y$, which in view of the initial conditions of the transport equations \eqref{transportodd}--\eqref{transporteven} means that the expression above reduces to 
\begin{equation}\label{Wronskian1+}
\mathcal W(u^+(x;x_0,\tilde x),u^-(x;x_0,\tilde y))
=4iw_\mathrm{even}^+(\tilde y,h;\tilde x).
\end{equation}
We may of course also choose $x=\tilde x$, thus obtaining
\begin{equation}\label{Wronskian1-}
\mathcal W(u^+(x;x_0,\tilde x),u^-(x;x_0,\tilde y))
=4iw_\mathrm{even}^-(\tilde x,h;\tilde y).
\end{equation}
In particular, we see that if there is a path $\Gamma(\tilde x,\tilde y)$ from $\tilde x$ to $\tilde y$ along which the function $t\mapsto \re z(t)$ is strictly increasing, then $\mathcal W(u^+(x;x_0,\tilde x),u^-(x;x_0,\tilde y))
=4i+O(h)$ as $h\to0$ by Remark \ref{h-asymptotic}. 
This shows that such a pair of solutions is linearly independent if $h$ is sufficiently small.

We shall also need the following formula, obtained by elementary computations, for pairs of solutions of the same type: 
\begin{multline}\label{Wronskian2}
\mathcal W(u^\pm(x;x_0,\tilde x),u^\pm(x;x_0,\tilde y))
=-4i(-1)^{(1\pm 1)/2}e^{\pm2 z(x;x_0)/h}\\ \times\big(
w_\mathrm{even}^\pm(x,h;\tilde x) w_\mathrm{odd}^\pm(x,h;\tilde y)-
w_\mathrm{even}^\pm(x,h;\tilde y)w_\mathrm{odd}^\pm(x,h;\tilde x)\big).
\end{multline}
In particular, if we can choose $x=x_\pm$ so that $\pm\re z(x_\pm;x_0)<0$ it follows that 
$\mathcal W(u^\pm(x;x_0,\tilde x),u^\pm(x;x_0,\tilde y))
=O(e^{-c/h})$ as $h\to0$ for some $c>0$.

\subsection{Conjugation}

Let $\bar x$ denote the scalar complex conjugate of $x\in\C$. 
For a matrix $A\in \mathcal L(\C^n,\C^m)$, let $\bar A$ denote the matrix with complex conjugated entries, so that $A^\ast={}^{t\mkern-4.5mu}\bar A$ is the conjugate transpose (adjoint).
For clarity we shall in this subsection write $u(x,\lambda)$ for a solution to \eqref{system}.
By using \eqref{exactWKB} and \eqref{transportodd}--\eqref{transporteven} it is straightforward to check (see the proof of Lemma \ref{rewritinglemma} below) that the WKB solutions enjoy the symmetry property
\begin{equation}\label{eq:WKBconjugation}
u^\pm(x;x_0,y_0,\lambda)=i\bigg( \begin{array}{cc} 0 & 1\\ -1 & 0\end{array} \bigg) \overline{u^\mp(\bar x;\bar x_0,\bar y_0,\bar\lambda)}.
\end{equation}
This implies that it suffices to study the spectral problem $P(h)u=\lambda u$ for spectral parameter $\lambda$ with nonnegative imaginary part.

\begin{lemma}\label{symmetrylemma}
The set of eigenvalues of $P(h)$ is symmetric with respect to the real axis.
\end{lemma}

\begin{proof}
Let $\lambda$ be an eigenvalue with $2\pi$-periodic eigenvector $u=u(x,\lambda)$. Introduce a pair of linearly independent WKB solutions $u^\pm(x,\lambda)$. Since the solution space of \eqref{system} is $2$-dimensional we have $u(x,\lambda)=c_1u^+(x,\lambda)+c_2u^-(x,\lambda)$, $c_j\in\C$.
Set 
\begin{equation*}
v(x)=\bigg(\begin{array}{cc} 0& -i\\i&0\end{array}\bigg)\overline{u(x,\lambda)}.
\end{equation*}
Then $v$ is $2\pi$-periodic. Moreover, \eqref{eq:WKBconjugation} 
implies that $v(x)=\bar c_1 u^-(x,\bar\lambda)+\bar c_2u^+(x,\bar\lambda)$ for $x\in\R$. Hence, $P(h)v=\bar\lambda v$, which completes the proof.
\end{proof}

\subsection{Periodic solutions}\label{sub:persol}

As a final preparation, we record a tractable condition for the existence of a nontrivial periodic solution of \eqref{system}, i.e., an  eigenvector of $P(h)$ corresponding to $\lambda$.

\begin{proposition}\label{prop:persol}
Let $u$ and $v$ be a pair of linearly independent solutions of \eqref{system},
and set $\tilde u(x)=u(x-2\pi)$ and $\tilde v(x)=v(x-2\pi)$.
Let $T$ be the transition matrix given by
\begin{equation}\label{generalT}
(u\ v)=(\tilde u\ \tilde v)T,
\end{equation}
where $(u\ v)$ is the $2\times2$ system with columns $u$ and $v$. Then $\det(T)=1$, and the existence of a nontrivial periodic solution of \eqref{system} is equivalent to the condition 
\begin{equation}\label{eq:transitioncondition}
\tr (T)=2.
\end{equation}
\end{proposition}

\begin{proof}
Taking the determinant of both sides of \eqref{generalT} we get
\begin{equation*}
\mathcal W(u,v)=\mathcal W(u({\hdot}-2\pi),v({\hdot}-2\pi))\det(T),
\end{equation*}
and since the Wronskian is independent of $x$ it follows that $\det(T)=1$. Hence, \eqref{eq:transitioncondition} is equivalent to $\tr(T)=\det(T)+1$, i.e.,
$\det(T-\id)=0$, which holds if and only if $Tc=c$ for some vector $c\ne0$. If $Tc=c$ then a simple computation shows that $x\mapsto (u(x)\ v(x))c$ is nontrivial and $2\pi$-periodic. Conversely, if $U$ is a nontrivial $2\pi$-periodic solution then $U$ can be expressed as a linear combination $U=c_1 u+c_2 v$, and the same computation as before shows that $c={}^{t\mkern-1mu}(c_1,c_2)$ satisfies $Tc=c$.
\end{proof}

\section{Eigenvalues in the absence of real turning points}\label{section:proofA}

Here we prove Theorem \ref{thm:A} by computing the trace of the transition matrix $T$ introduced above, then applying  \eqref{eq:transitioncondition} and analyzing the result.
We fix $\lambda_0$ satisfying the hypotheses of Theorem \ref{thm:A}, then $V(x)^2+\lambda_0^2>0$ for all $x\in\R$ and there are no turning points on the real axis.
Choose a determination of 
\begin{equation*}
z(x;x_0)=z(x;x_0,\lambda_0)=i\int_{x_0}^x\sqrt{V(t)^2+\lambda_0^2}\de t
\end{equation*}
by picking the branch of the square root satisfying $(V(x)^2+\lambda_0^2)^{1/2}>0$ at $x=0$. 
Clearly, $\re z(x)$ is independent of $x\in\R$, so the real axis is a Stokes line.
Since Stokes lines cannot intersect 
(see e.g.~\cite{evgrafov1966asymptotic}) we find by restricting to a sufficiently small tubular neighborhood of $\R$ (which in particular should contain no turning points) that the Stokes lines are essentially parallel to the real axis there.
Recall that away from turning points, the configuration of Stokes lines depends continuously on the parameter $\lambda$. Hence, we can find $\ve>0$ such that if $|\lambda-\lambda_0|<\ve$ then there is a turning-point-free neighborhood of the real axis in which the imaginary axis is still transversal to the tangent vectors of any Stokes line.
From now on, we fix $\lambda\in B_\ve(\lambda_0)$. We also choose $\ve$ so small that $(V(x)^2+\lambda^2)^{1/2}$ has positive real part at $x=0$. This gives a determination of $z(x)$ for this choice of $\lambda$ which is consistent with the determination of $z(x;x_0,\lambda_0)$ above.

Recall that the number of Stokes lines starting from a turning point $x_0$ of order $n$ is $n+2$. If $W(x)=-V(x)^2$, the Stokes lines starting from $x_0$ have argument
\begin{equation}\label{eq:simpledoubleargs}
\begin{aligned}
& \pi/3-\tfrac{1}{3}\Arg{W'(x_0)}\text{ mod $2\pi/3$ if $x_0$ is simple and}\\
& \pi/4-\tfrac{1}{4}\Arg{W''(x_0)}\text{ mod $\pi/2$ if $x_0$ is double,}
\end{aligned}
\end{equation}
see the study by G{\'e}rard and Grigis \cite[p.~152]{gerard1988precise}.

\begin{example}\label{ex:1}
Let $V(x)=\cos x$. For $\lambda>0$, the turning points are simple, and can be found by solving $\zeta+\zeta^{-1}=2 i (-1)^j \lambda$ for $j=1,2$, where $\zeta=e^{ix}$.
Upon taking logarithms we get
\begin{equation*}
x_\pm^{(j)}=\pm \Big(\pi/2- i(-1)^j\log\Big(1+\sqrt{1+\lambda^2}\Big)\Big)+2\pi \Z.
\end{equation*}
Here, $x_+^{(1)}$ and $x_-^{(2)}$ lie in the upper half plane, and $x_-^{(1)}$ and $x_+^{(2)}$ in the lower, and it is easy to check that
\begin{equation*}
\Arg W'(x_\pm^{(j)})=\pm(-1)^j\pi/2.
\end{equation*}
Hence, Stokes lines starting from turning points in the upper half plane have arguments $\pi/2$ mod $2\pi/3$. Stokes lines starting from turning points in the lower half plane have arguments $-\pi/2$ mod $2\pi/3$. Figure \ref{figure1} shows the configuration of Stokes lines for $\lambda=1$ and $\lambda=1+10^{-1}i$.
\end{example}

\begin{figure}
	\centering
\begin{tikzpicture}[scale=.75]
\begin{axis}[
    y={0.16\linewidth},
    x={0.06\linewidth},
    minor xtick={-3.1416, -2.7489, -2.3562, -1.9635, -1.1781, -0.78539, -0.3927, 0, 0.3927, 0.78539, 1.1781, 1.9635, 2.3562, 2.7489, 3.1416},
    minor ytick={-0.8,-0.7,-0.6,-0.4,-0.3,-0.2,-0.1,0.1,0.2,0.3,0.4,0.6,0.7,0.8},
    xtick      ={-3.1416, -1.5708, 0, 1.5708, 3.1416},
        xticklabels={$-\pi$, $-\pi/2$, $0$, $\pi/2$, $\pi$},    
        ytick={-0.5,0,0.5},
xmin=-3.5,
xmax=3.5,
ymin=-0.925,
ymax=0.95,
xlabel={$\re x$},
ylabel={$\im x$},
enlarge x limits=0.005,
enlarge y limits=0.005,
]
\addplot graphics[xmin=-3.27,ymin=-0.821,xmax=3.27,ymax=0.836]{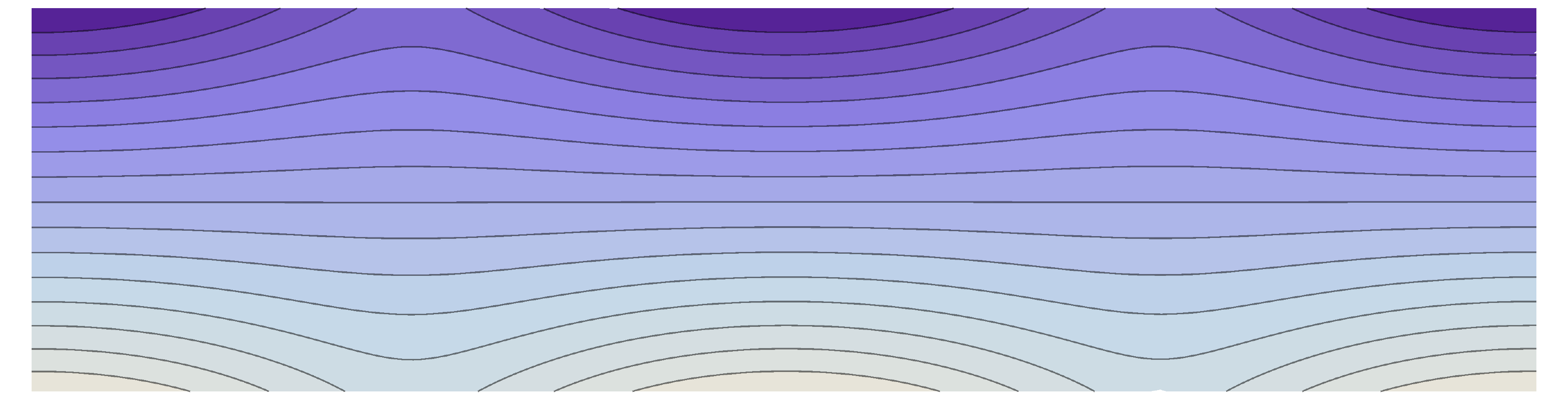};
\end{axis}
\end{tikzpicture}
\quad
\begin{tikzpicture}[scale=.75]
\begin{axis}[
    y={0.16\linewidth},
    x={0.06\linewidth},
    minor xtick={-3.1416, -2.7489, -2.3562, -1.9635, -1.1781, -0.78539, -0.3927, 0, 0.3927, 0.78539, 1.1781, 1.9635, 2.3562, 2.7489, 3.1416},
    minor ytick={-0.8,-0.7,-0.6,-0.4,-0.3,-0.2,-0.1,0.1,0.2,0.3,0.4,0.6,0.7,0.8},
    xtick      ={-3.1416, -1.5708, 0, 1.5708, 3.1416},
        xticklabels={$-\pi$, $-\pi/2$, $0$, $\pi/2$, $\pi$},    
        ytick={-0.5,0,0.5},
xmin=-3.5,
xmax=3.5,
ymin=-0.925,
ymax=0.95,
xlabel={$\re x$},
ylabel={$\im x$},
enlarge x limits=0.005,
enlarge y limits=0.005,
]
\addplot graphics[xmin=-3.27,ymin=-0.821,xmax=3.27,ymax=0.836]{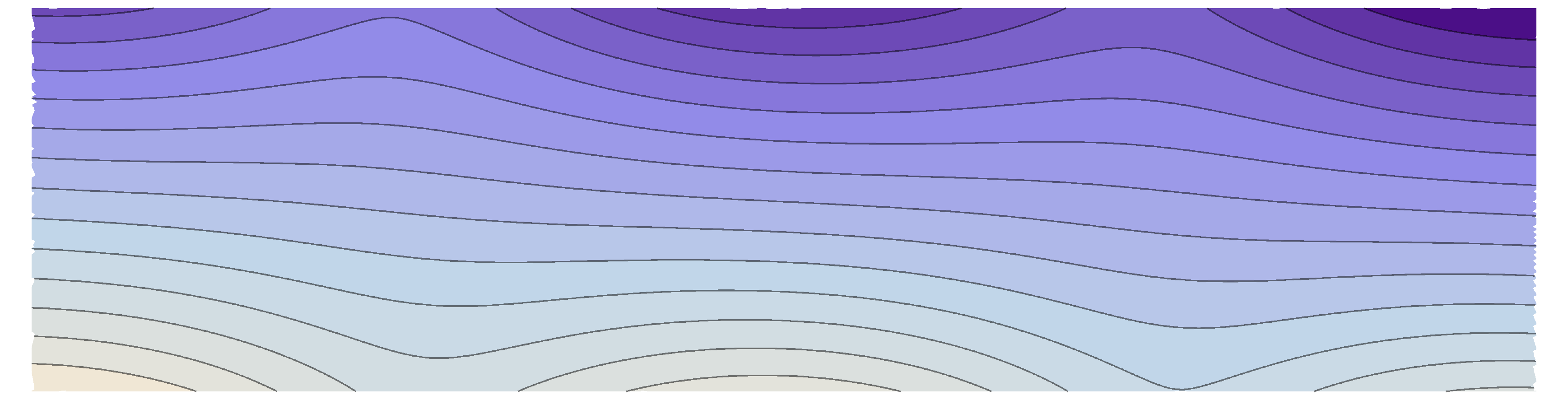};
\end{axis}
\end{tikzpicture}
\quad
\includegraphics[width=0.075\linewidth]{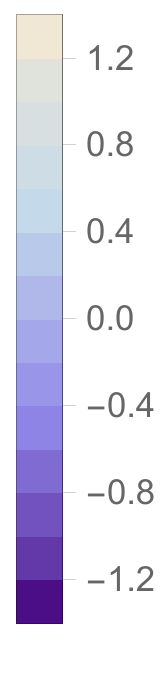}
\caption{\label{figure1} The configuration of Stokes lines for $V(x)=\cos x$, with legends describing the size of $\re z(x;0,\lambda)$ for $\lambda=1$ (left panel) and $\lambda=1+10^{-1} i$ (right panel).}
\end{figure}

Note that if we let $\lambda\to0$ along the real line in Example \ref{ex:1}, then the turning points collapse onto the real line to form double turning points, so that case is excluded from Theorem \ref{thm:A}. For example, $x_+^{(1)}$ and $x_+^{(2)}$ collapse to $\pi/2$. The Stokes lines starting from the resulting turning points have argument 0 mod $\pi/2$.

We now return to the general situation treated in Theorem \ref{thm:A}.
For $x$ near 0 we get by Taylor's formula 
\begin{equation}\label{Taylor's}
z(x;x_0)-z(0;x_0)=ix(V(0)^2+\lambda^2)^{1/2}(1+g(x)),
\end{equation}
where $g$ is analytic and $g(0)=0$. Moreover, $(V(0)^2+\lambda^2)^{1/2}$ is approximately real by assumption. Hence, for $x$ near 0 we have $\re z(x)\simeq \re z(0)-\im x$, showing that $\re z(x)$ is a strictly decreasing function of $\im x$ for $x$ near 0. This remains true in a small tubular neighborhood of the real axis, since there, any line parallel to the imaginary axis is transversal to the Stokes lines. The reader is asked to compare with Figure \ref{figure1}.

Let $y_0$ be an amplitude base point in the upper half plane near the real axis with $\re y_0=0$, and let $\bar y_0$ be the complex conjugate. We will choose phase base points on the real line. With the previous discussion in mind, and because we want our WKB solutions to have asymptotic formulas valid in intersecting domains, we introduce  the four WKB solutions
\begin{align*}
u_0^+(x)&=u^+(x;0,y_0), & u_1^+(x)&=u^+(x;2\pi,y_0+2\pi),\\
u_0^-(x)&=u^-(x;0,\bar y_0), & u_1^-(x)&=u^-(x;2\pi,\bar y_0+2\pi),
\end{align*}
defined in accordance with \eqref{u}.
Inspecting the definition (see \eqref{transportodd}--\eqref{transporteven}) we find that $u_1^\pm(x)=u_0^\pm(x-2\pi)$, so Proposition \ref{prop:persol} is applicable.

\begin{proposition}\label{prop:Tsimplecase}
Let $T$ be the transition matrix defined by $(u_0^+\ u_0^-)=(u_1^+\ u_1^-)T$ and let $I$ denote the action integral $I(\lambda)=\int_{0}^{2\pi}(V(t)^2+\lambda^2)^{1/2}\de t$. Then
\begin{equation*}
T=\bigg( \begin{array}{cc} e^{iI/h}& 0\\ 0& e^{-iI/h}\end{array} \bigg)\bigg( \begin{array}{cc} t_{11}& t_{12}\\ t_{21}& t_{22}\end{array} \bigg),
\end{equation*}
where $t_{11}t_{22}-t_{12}t_{21}=1$. Moreover, $t_{jj}=1+r_j(\lambda,h)$ where $r_j$ depends holomorphically on $\lambda$, and $r_j=O(h)$ while $t_{12},t_{21}=O(e^{-\delta/h})$ for some $\delta>0$ as $h\to0$.
\end{proposition}

\begin{proof}
Introduce the auxiliary solutions
\begin{align*}
\widetilde u_1^+(x)&=u^+(x;0,y_0+2\pi),\\
\widetilde u_1^-(x)&=u^-(x;0,\bar y_0+2\pi),
\end{align*}
which differ from $u_1^+$ and $u_1^-$ only in the choice of phase base point. If $\widetilde T$ is the transition matrix defined by $(u_0^+\ u_0^-)=(\widetilde u_1^+\ \widetilde u_1^-)\widetilde T$,
then a simple computation gives 
\begin{equation*}
T=\bigg( \begin{array}{cc} e^{iI/h}& 0\\ 0& e^{-iI/h}\end{array} \bigg)\widetilde T,
\end{equation*}
with $I$ given above.

We next determine $\widetilde T=(t_{ij})$, noting that $t_{11}t_{22}-t_{12}t_{21}=1$ since $\det(T)=1$ by Proposition \ref{prop:persol}. Taking Wronskians we get
\begin{align*}
t_{11}&=\frac{\mathcal W(u_0^+,\widetilde u_1^-)}{W(\widetilde u_1^+,\widetilde u_1^-)}, & t_{12}&=\frac{\mathcal W(u_0^-,\widetilde u_1^-)}{W(\widetilde u_1^+,\widetilde u_1^-)},\\
t_{21}&=\frac{\mathcal W(\widetilde u_1^+,u_0^+)}{W(\widetilde u_1^+,\widetilde u_1^-)}, & t_{22}&=\frac{\mathcal W(\widetilde u_1^+,u_0^-)}{W(\widetilde u_1^+,\widetilde u_1^-)}.
\end{align*}
We first compute $\mathcal W(\widetilde u_1^+,\widetilde u_1^-)$. By the properties of $\re z(x)$, we can find a curve from $y_0+2\pi$ (in the upper half plane) to $\bar y_0+2\pi$ (in the lower half plane) along which $\re z(x)$ is strictly increasing. Evaluating the Wronskian at $\bar y_0+2\pi$ (see \eqref{Wronskian1+}) 
we obtain
\begin{equation*}
\mathcal W(\widetilde u_1^+,\widetilde u_1^-)=4iw^+_\mathrm{even}(\bar y_0+2\pi,h;y_0+2\pi),
\end{equation*}
which is $4i+O(h)$ by Remark \ref{h-asymptotic}.
Since we can also find curves from $y_0$ to $\bar y_0+2\pi$ and from $y_0+2\pi$ to $\bar y_0$ along which $\re z(x)$ is strictly increasing, the same arguments show that 
\begin{gather*}
\mathcal W(u_0^+,\widetilde u_1^-)=4iw^+_\mathrm{even}(\bar y_0+2\pi,h;y_0),\\
\mathcal W(\widetilde u_1^+,u_0^-)=4iw^+_\mathrm{even}(\bar y_0,h;y_0+2\pi),
\end{gather*}
and both are equal to $4i+O(h)$. Since $w^+_\mathrm{even}$ depends analytically on $\lambda$ we find that 
$t_{jj}=1+r_j$, where $r_j=r_j(\lambda,h)$ is analytic in $\lambda$ and $r_j=O(h)$ as $h\to0$.

For $\mathcal W(\widetilde u_1^+,u_0^+)$ and $\mathcal W(u_0^-,\widetilde u_1^-)$, we use the Wronskian formula \eqref{Wronskian2} for solutions of the same type. For $\mathcal W(\widetilde u_1^+,u_0^+)$, we evaluate \eqref{Wronskian2} at the point $x_1=i\eta$ above 0 for some small $\eta>0$. Then $\re z(x_1;0)<\re z(0;0)=0$, so $e^{z(x_1;0)/h}$ is exponentially decreasing as $h\to0$. For $\mathcal W(u_0^-,\widetilde u_1^-)$, we evaluate \eqref{Wronskian2} at the point $\bar x_1=-i\eta$ below 0. 
Then $\re -z(\bar x_1;0)<\re z(0;0)=0$, so $e^{-z(\bar x_1;0)/h}$ is exponentially decreasing as $h\to0$. Hence,
\begin{equation*}
t_{12},t_{21}=O(e^{-\delta/h}),\quad h\to0,
\end{equation*}
for some $\delta>0$, which completes the proof. 
\end{proof}

\begin{proof}[End of Proof of Theorem \ref{thm:A}]
By Proposition \ref{prop:persol} it follows that $\lambda\in B_\ve(\lambda_0)$ is an eigenvalue of $P(h)$ if and only if $\tr(T)=2$, i.e.,
\begin{equation*}
e^{iI/h}t_{11}-2+e^{-iI/h}t_{22}=0,
\end{equation*}
which is easily seen to yield \eqref{eq:theoremA}.
Multiplying with $e^{iI/h}/t_{11}$ and completing the square, an elementary computation using $t_{11}t_{22}-t_{12}t_{21}=1$ gives
\begin{equation*}
\bigg(e^{iI/h}-\frac{1+ic}{t_{11}}\bigg)\bigg(e^{iI/h}-\frac{1-ic}{t_{11}}\bigg)=0,
\end{equation*}
where $c=(t_{12}t_{21})^{1/2}$. 
Hence, $e^{iI/h}=1+O(h)$.
Taking logarithms we conclude that there is an integer $k\in\Z$ such that
\begin{equation*}
iI/h-2\pi i k=\log(1+O(h))=O(h).
\end{equation*}
This gives the desired quantization condition \eqref{quantizationcondition1}.
\end{proof}

\section{Eigenvalues in the presence of real turning points}\label{section:proofB}

We now turn to the proof of Theorem \ref{thm:B}, and we let $\lambda_0=i\mu_0$ with $\mu_0\in (V_1,V_0)$ be fixed. By assumption, all turning points for $\lambda_0$ are simple.
As in Section \ref{section:proofA} we begin by describing the configuration of Stokes lines for $\lambda_0$ before turning to general parameter values $\lambda$ close to $\lambda_0$.

\subsection{Turning points} 
As described in the introduction there are $2l$ turning points
\begin{equation*}
0<x_1(\mu_0)<\ldots<x_{2l}(\mu_0)<2\pi.
\end{equation*}
Since $V(0)=V_0$ is a local maximum, $x_1(\mu_0)$ must be a simple zero of $V-\mu_0$. Thus, $V'(x_1(\mu_0))< 0$ and by \eqref{eq:simpledoubleargs} we have 
\begin{equation*}
\Arg{W'(x_1(\mu_0))}=\Arg{\left[-2\mu_0 V'(x_1(\mu_0))\right]}=0,
\end{equation*}
so the Stokes lines starting at $x_1(\mu_0)$ have arguments $\pi/3$ mod $2\pi/3$. Basic calculus shows that $\Arg{W'(x_2(\mu_0))}=\pi$, so the Stokes lines starting at $x_2(\mu_0)$ have arguments $0$ mod $2\pi/3$. A moments reflection shows that this pattern repeats itself, with the Stokes lines starting at odd numbered turning points having arguments $\pi/3$ mod $2\pi/3$, and the Stokes lines starting at even numbered turning points having arguments $0$ mod $2\pi/3$; by periodicity, this includes the last turning point $x_{2l}(\mu_0)$ on $(0,2\pi)$. In particular, we see that there are bounded Stokes lines lying on $\R$ starting from even numbered turning points (on the left) and ending at odd numbered turning points (on the right). However, note that this is not a stable configuration and will not persist when $\lambda_0$ is perturbed off the imaginary axis.

We take
\begin{equation*}
z(x;x_j(\mu_0),\lambda_0)=i\int_{x_j(\mu_0)}^x(V(t)^2 -\mu_0^2)^{1/2}\de t, \quad\lambda_0=i\mu_0,
\end{equation*}
where the choice of turning point $x_j(\mu_0)$ will depend on the domain of interest, see \eqref{2pairs}--\eqref{1tildepair} below.
We define the Riemann surfaces of $z(x)$ and $H(z(x))$ over $D$ by introducing branch cuts from odd numbered turning points  along the Stokes lines with argument $-\pi/3$, and from even numbered turning points along the Stokes lines with argument $2\pi/3$.
We choose branches so that 
\begin{equation*}
H(z(0);\lambda_0)=\bigg(\frac{V(0)+\mu_0}{V(0)-\mu_0}\bigg)^{1/4}>0, \quad\lambda_0=i\mu_0.
\end{equation*}
Since $V(0)\pm\mu_0>0$, this also gives a determination of $(V^2-\mu_0^2)^{1/2}$ and therefore of $z(x)$, namely, $(V^2-\mu_0^2)^{1/2}>0$ at the origin. By applying \eqref{Taylor's} with $\lambda_0=i\mu_0$ we see that $\re z(x)$ is a strictly decreasing function of $\im x$ for $x$ near 0. This determines the behavior of $\re z(x)$ in any simply connected open set that intersects the imaginary axis, does not contain any turning points, and does not pass through a branch cut. In particular, there is also for each $j\ge1$ a region between $x_{2j}(\mu_0)$ and $x_{2j+1}(\mu_0)$ where $\re z(x)$ is a strictly decreasing function of $\im x$.

\begin{figure}
	\centering
\begin{tikzpicture}[scale=.75]
\begin{axis}[
    y={0.22\linewidth},
    x={0.22\linewidth},
   minor xtick={-1.37445, -1.1781, -0.98175, -0.58905,-0.3927, -0.19635, 0.19635, 0.3927, 0.58905, 0.98175, 1.1781, 1.37445, 1.76715, 1.9635, 2.15985, 2.55255, 2.7489, 2.94525},
   minor ytick={-0.4,-0.3,-0.2,-0.1,0.1,0.2,0.3,0.4},
   xtick      ={-0.78539, 0, 0.78539, 1.5708, 2.3562},
        xticklabels={$-\pi/4$, $0$, $\pi/4$, $\pi/2$, $3\pi/4$},    
            ytick={-0.5,0,0.5},
xmin=-1.6,
xmax=3.1,
ymin=-0.6,
ymax=0.6,
xlabel={$\re x$},
ylabel={$\im x$},
enlarge x limits=0.005,
enlarge y limits=0.005,
]
\addplot graphics[xmin=-1.57,ymin=-0.523,xmax=3.075,ymax=0.523]{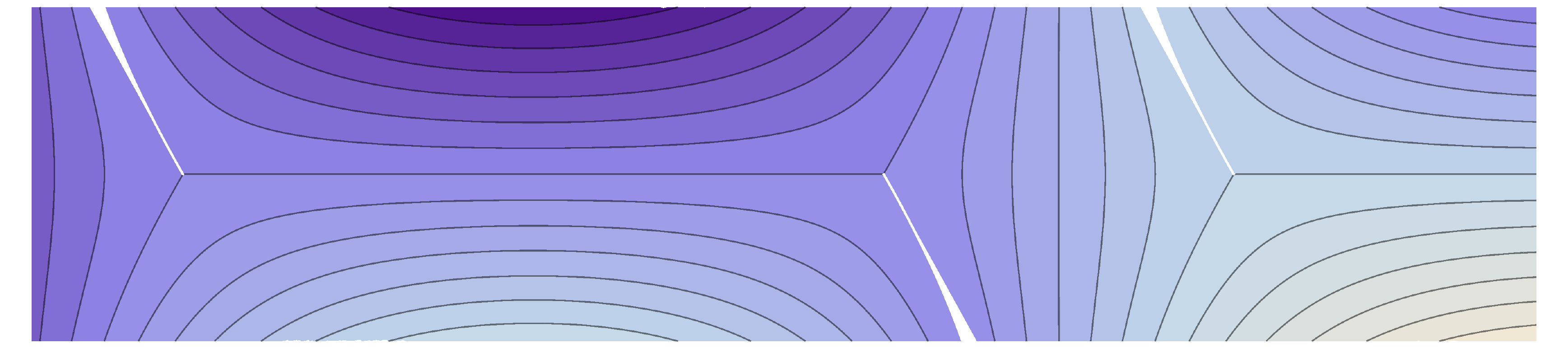};
\end{axis}
\end{tikzpicture}
\quad
\includegraphics[width=0.07\linewidth]{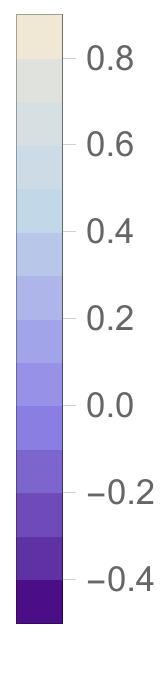}
\begin{tikzpicture}[scale=.75]
\begin{axis}[
    y={0.22\linewidth},
    x={0.22\linewidth},
   minor xtick={-1.37445, -1.1781, -0.98175, -0.58905,-0.3927, -0.19635, 0.19635, 0.3927, 0.58905, 0.98175, 1.1781, 1.37445, 1.76715, 1.9635, 2.15985, 2.55255, 2.7489, 2.94525},
   minor ytick={-0.4,-0.3,-0.2,-0.1,0.1,0.2,0.3,0.4},
   xtick      ={-0.78539, 0, 0.78539, 1.5708, 2.3562},
        xticklabels={$-\pi/4$, $0$, $\pi/4$, $\pi/2$, $3\pi/4$},    
            ytick={-0.5,0,0.5},
xmin=-1.6,
xmax=3.1,
ymin=-0.6,
ymax=0.6,
xlabel={$\re x$},
ylabel={$\im x$},
enlarge x limits=0.005,
enlarge y limits=0.005,
]
\addplot graphics[xmin=-1.57,ymin=-0.523,xmax=3.075,ymax=0.523]{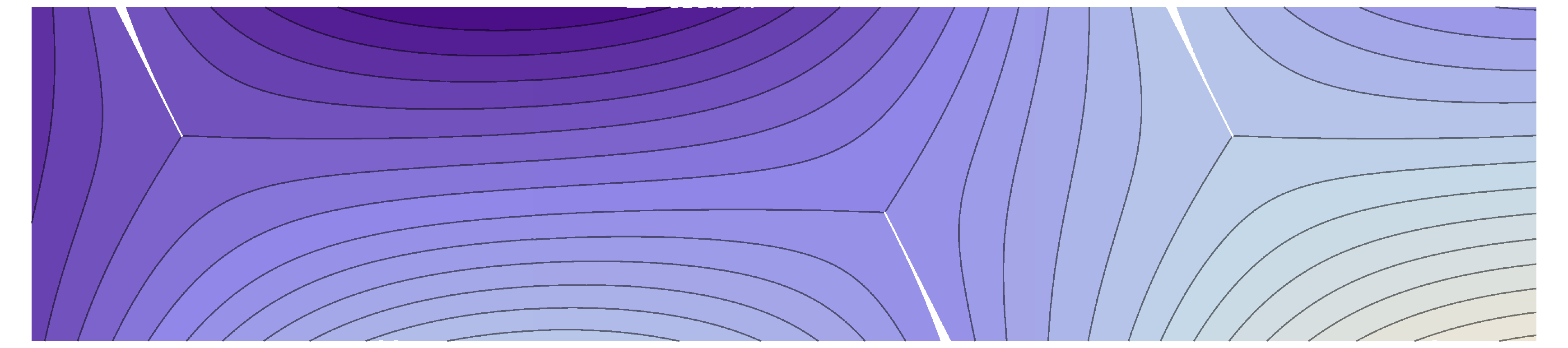};
\end{axis}
\end{tikzpicture}
\quad
\includegraphics[width=0.07\linewidth]{hard_legends.pdf}
\caption{\label{figure2} The configuration of Stokes lines and the behavior of $\re z(x;x_1,\lambda)$ for $V(x)=\cos x$, 
where $\lambda=i\mu$ and $x_1$ satisfies $\re x_1>0$, $\cos x_1=\mu$ (the turning point in the middle). The top panel describes the situation for $\mu=1/2$ and the bottom panel for $\mu=1/2+ i/10$. Branch cuts are located along (the curved edges of) the white regions.}
\end{figure}

When $\lambda_0$ is perturbed through rotation around the origin, the turning points are rotated around points on the real axis, e.g., $x_1$ and $x_2$ are rotated around the mid point $(x_1+x_2)/2$, see \cite{MR1635811}. Each bounded Stokes line lying on $\R$ will then split into two unbounded Stokes lines, but with the topology of the Stokes configuration otherwise unchanged. If the perturbation is allowed to continue, then at a rotation angle of around $\pi/4$ new bounded Stokes lines will appear between simple turning points,
and these lines will coincide with the integration paths of the action integrals $S_j(\mu)$ defined by \eqref{actionS}.
The first change is not significant for the proof of Theorem \ref{thm:B}, while the second change completely alters the behavior of $\re z(x)$ and is {\it not} permitted.
From now on we therefore fix $\ve>0$ such that the integration paths of $S_j(\mu)$ are not bounded Stokes lines when $\lambda\in B_\ve(\lambda_0)$, and we make sure that
\begin{equation}\label{positiverealpart}
\lambda\in B_\ve(\lambda_0)\quad\Longrightarrow\quad\re S_j(\mu)>0,\quad 1\le j\le l.
\end{equation}
We also take $\ve$ so small that if $\lambda\in B_\ve(\lambda_0)$ is purely imaginary, then the turning points are simple and the Stokes configuration is the same as for $\lambda_0$; in particular this means that $B_\ve(\lambda_0)$ does not contain $0$ and $i V_0$. Moreover, the arguments of Stokes lines at the turning points will be almost unchanged, so we place branch cuts as described for $\lambda=\lambda_0$, modified in the obvious manner. We also use the inherited determination of $H(z({\hdot}))$ and $(V^2-\mu^2)^{1/2}$, namely the one which has positive real part at the origin. 
Figure \ref{figure2} illustrates the configuration of Stokes lines near three consecutive turning points, including the location of branch cuts.

\subsection{The transition matrix}
Under the assumptions of Theorem \ref{thm:B}, the transition matrix $T$, as defined in Proposition \ref{prop:persol}, will consist of a product of intermediate transition matrices. These matrices will be of (at most) four types, which we now describe.

Let $\lambda=i\mu\in B_\ve(\lambda_0)$.
We fix amplitude base points $y_1,\ldots, y_l\in D$ in the upper half plane independent of $\lambda$ as in Figure \ref{figure3}
in such a way that $x_{2j}(\mu)<\re y_j<x_{2j+1}(\mu)$ and so that $ y_j$ is always in the
same region bounded by Stokes lines when $\lambda$ varies in $B_\varepsilon(\lambda_0)$. We then set $y_0=y_l-2\pi$.
Introduce the WKB solutions
\begin{equation}\label{2pairs}
u_j^+(x)=u^+(x;x_{2j+1}(\mu),y_j),\quad
u_j^-(x)=u^-(x;x_{2j+1}(\mu),\bar y_j),
\end{equation}
defined for $j=0,\ldots,l$ in accordance with \eqref{u}, where $u_l^\pm(x)=u_0^\pm(x-2\pi)$.
The transition from the pair $u^+_{j-1},u^-_{j-1}$ to the pair $u_j^+,u_j^-$ is one of the following four kinds:

\begin{itemize}
\item[$1^\circ$.] Both $x_{2j-1}(\mu)$ and $x_{2j}(\mu)$ are zeros of $V(x)-\mu$.
\item[$2^\circ$.] $x_{2j-1}(\mu)$ is a zero of $V(x)-\mu$ and $x_{2j}(\mu)$ is a zero of $V(x)+\mu$.
\item[$3^\circ$.] $x_{2j-1}(\mu)$ is a zero of $V(x)+\mu$ and $x_{2j}(\mu)$ is a zero of $V(x)-\mu$.
\item[$4^\circ$.] Both $x_{2j-1}(\mu)$ and $x_{2j}(\mu)$ are zeros of $V(x)+\mu$.
\end{itemize}
That there are no other kinds follows from the fact that since $V(0)=V_0$, the nature of the zeros $x_{2j-1}$ and $x_{2j}$ also determines the nature of $x_{2j+1}$, e.g., if both $x_{2j-1}$ and $x_{2j}$ are zeros of $V(x)-\mu$ then so is $x_{2j+1}$.
Introduce also the auxiliary solutions
\begin{equation}\label{1tildepair}
\widetilde u_j^+(x)=u^+(x;x_{2j}(\mu),y_j),\quad
\widetilde u_j^-(x)=u^-(x;x_{2j}(\mu),\bar y_j), \quad j=1,\ldots,l.
\end{equation}

\definecolor{qqqqff}{rgb}{0.368417, 0.506779, 0.709798}
\begin{figure}
	\centering
\begin{tikzpicture}[scale=.38]
\draw[color=qqqqff,dashed,thick] 
  (1.5,4.65) 
    .. controls (0.71,6.1) and (0.56,6.6) .. 
  (0.3,7.6) ; 
\draw[color=qqqqff,thick]
  (1.5,4.65) 
    .. controls (1.26,4.3) and (0.55,3) .. 
  (0.33,2.4) ;
\draw[color=qqqqff,thick]
  (1.5,4.65)
    .. controls (6.16,4.48) and (7.48,5.42) .. 
  (8.27,8.25)   ;
\draw[color=qqqqff,thick]
   (0.32,0.2)
    .. controls (1.1,2.83) and (2.16,3.96) .. 
  (7.13,3.75) ;
\draw[color=qqqqff,thick]
  (7.13,3.75) 
    .. controls (8,5.2) and (8.75,6.5) .. 
  (8.84,8.25) ; 
\draw[color=qqqqff,dashed,thick] 
   (8.4,0.2)
    .. controls (8.2,1.5) and (7.77,2.6) .. 
  (7.13,3.75) ;
\draw[color=qqqqff,thick]
   (23.68,8.25)  
    .. controls (21.9,5.67) and (20.84,4.54) ..  
  (12.87,4.75) ;
\draw[color=qqqqff,thick]
  (12.87,4.75) 
    .. controls (12,3.3) and (11.25,2) .. 
  (11.16,0.25) ; 
\draw[color=qqqqff,dashed,thick] 
   (11.6,8.25)
    .. controls (11.8,7) and (12.23,5.9) .. 
  (12.87,4.75) ;
\draw[color=qqqqff,thick]
  (23.33,3.95)
    .. controls (14.84,4.02) and (13.52,3.08) .. 
  (12.3,0.2)   ;
\draw[color=qqqqff,dashed,thick] 
   (24.6,0.4)
    .. controls (24.4,1.7) and (23.97,2.8) .. 
  (23.33,3.95) ;
\draw[color=qqqqff,thick]
  (23.33,3.95) 
    .. controls (23.8,4.72) and (24.35,5.6) .. 
  (24.6,6.24)  ; 
	\filldraw (1.5,4.65) circle[radius=3pt] node[left] {$x_{2j}$};
	\filldraw (7.13,3.75) circle[radius=3pt] node[right] {$x_{2j+1}$};
	\filldraw (12.87,4.75) circle[radius=3pt] node[left] {$x_{2j+2}$};
	\filldraw (23.33,3.95) circle[radius=3pt] node[right] {$x_{2j+3}$};
	\filldraw (4.315,8) circle[radius=3pt] node[below right] {$y_{j}$};
	\filldraw (4.315,0.4) circle[radius=3pt] node[below right] {$\bar y_{j}$};
	\filldraw (18.1,8) circle[radius=3pt] node[below right] {$y_{j+1}$};
	\filldraw (18.1,0.4) circle[radius=3pt] node[below right] {$\bar y_{j+1}$};
  \end{tikzpicture}
\caption{\label{figure3} The location of amplitude base points relative the neighboring turning points $x_k(\mu)$ over a partial period for generic $V$ and $\lambda=i\mu\in B_\ve(\lambda_0)$. Branch cuts are indicated by dashed lines.}
\end{figure}
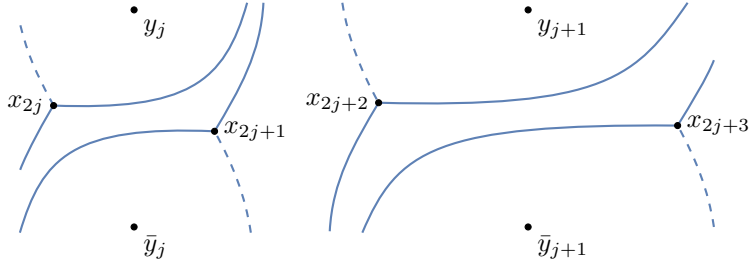

Let $T_j$ (resp.~$\widetilde T_j$) be the transition matrix between $u_{j-1}^+,u_{j-1}^-$ and $u_j^+,u_j^-$ (resp.  between $u_{j-1}^+,u_{j-1}^-$ and $\widetilde u_j^+,\widetilde u_j^-$):
\begin{equation}\label{tildeT_j}
(u_{j-1}^+\ u_{j-1}^-)=(u_j^+\ u_j^-)T_j, \quad (u_{j-1}^+\ u_{j-1}^-)=(\widetilde u_j^+\ \widetilde u_j^-)\widetilde T_j.
\end{equation}
For each for $j=1,\ldots, l$ it is clear that $\widetilde T_j$ is of the same transition type as $T_j$.
The transition matrix $T$ as defined in Proposition \ref{prop:persol} is given by
\begin{equation}\label{Tagain1}
T=T_l\cdots T_1.
\end{equation}

\begin{example}
In the case where $V(x)=\cos x$ and $\lambda=i\mu$ with $\mu\in(0,1)$, there are four turning points (i.e. $l=2$),
and $T=T_2T_1$. The matrix $T_1$ is of type $2^\circ$ and $T_2$ is of type $3^\circ$.
In the case where $V(x)=\tfrac{1}{3}(2+\cos x)$ and $\lambda=i\mu$ with $\mu\in(\tfrac{1}{3},1)$, there are two turning points (i.e. $l=1$),
and $T=T_1$. The matrix $T_1$ is of type $1^\circ$.
\end{example}

Recall the definition of the action integrals $I_j(\mu)$ in \eqref{actionS}. By straightforward computation we immediately obtain the following relationship between $T_j$ and $\widetilde T_j$. 
\begin{lemma}\label{Iaction}
\begin{equation*}
T_j=\bigg( \begin{array}{cc} e^{iI_j/h}& 0\\ 0& e^{-iI_j/h}\end{array} \bigg)\widetilde T_j.
\end{equation*}
\end{lemma}

We now describe the different types of transition matrices.

\begin{theorem}\label{thm:transitionALL}
 If $T_j$ is  of type $m^\circ$, then
\begin{equation*}
T_j(\lambda,h)=e^{S_j(\mu)/h}\bigg( \begin{array}{cc} e^{iI_j/h}& 0\\ 0& e^{-iI_j/h}\end{array} \bigg)\Big( E_m+R_j(\lambda,h)\Big),
\end{equation*}
where the matrix $R_j(\lambda,h)$ depends analytically on $\lambda$ in $B_\ve(\lambda_0)$ for some positive $\ve$ and is uniformly of $O(h)$ there as $h\to0$, and
\begin{equation*}
E_1=\bigg( \begin{array}{cc} 1& i\\ -i& 1\end{array} \bigg),  \
E_2=\bigg( \begin{array}{cc} 1& i\\ i& -1\end{array} \bigg), \
E_3=\bigg( \begin{array}{cc} 1& -i\\ -i& -1\end{array} \bigg), \ 
E_4=\bigg( \begin{array}{cc} 1& -i\\ i& 1\end{array} \bigg).
\end{equation*}
\end{theorem}

We postpone the proof of Theorem \ref{thm:transitionALL} to Section \ref{section:sheets}, where it will be an immediate consequence of Theorem \ref{thm:transition2} and Theorems \ref{thm:transition3}--\ref{thm:transition4} together with Lemma \ref{Iaction}.

\subsection{Computing the trace}

By \eqref{Tagain1}, $T$ is a product of matrices of transitions $1^\circ$ -- $4^\circ$, the precise nature of which depends on the potential $V$. However, in a sense made precise below, the trace of $T$ does not. 
Note that there are some restrictions on the possible factors in a product of intermediate transition matrices such as \eqref{Tagain1}: if we for the moment write $T_j^{(m)}$ to indicate that $T_j$ is a transition matrix of type $m^\circ$,
then \eqref{Tagain1} cannot contain any of the factors
\begin{equation*}
T^{(1)}_{j+1}T^{(4)}_j,\quad T^{(4)}_{j+1}T^{(1)}_j,\quad
T^{(3)}_{j+1}T^{(1)}_j,\quad T^{(1)}_{j+1}T^{(2)}_j,\quad 
T^{(2)}_{j+1}T^{(4)}_j,\quad T^{(4)}_{j+1}T^{(3)}_j
\end{equation*}
and by periodicity it must contain an equal number of matrices of transitions $2^\circ$ and $3^\circ$.

\begin{proposition}\label{prop:trace}
For $l\ge1$ let $T$ be given by \eqref{Tagain1}. Then
\begin{equation*}
\tr(T)=\exp\bigg(\sum_{j=1}^lS_j/h\bigg)\bigg(2^l\prod_{j=1}^{l}\cos(I_j/h)+r(\lambda,h)\bigg),
\end{equation*}
where $r(\lambda,h)$ depends analytically on $\lambda$ in $B_\ve(\lambda_0)$ and is uniformly of $O(h)$ there as $h\to0$.
\end{proposition}

\begin{proof}
For $k\le j$ we let 
$g(j,k)=2^{j-k}\prod_{n=k}^{j-1} \cos(I_n/h)$ with the convention that $g(j,j)=1$, and set 
\begin{equation*}
A_\pm(j,k)=g(j,k)\bigg(\begin{array}{cc}e^{iI_j/h}&0\\ 0&e^{-iI_j/h}\end{array}\bigg)\bigg(\begin{array}{cc}1&\pm i\\ \mp i&1\end{array}\bigg).
\end{equation*}
In view of Theorem \ref{thm:transitionALL}, the result follows if we show that
\begin{equation}\label{Tformula}
\bigg(\begin{array}{cc}e^{iI_l/h}&0\\ 0&e^{-iI_l/h}\end{array}\bigg) E_{m_l}\cdots \bigg(\begin{array}{cc}e^{iI_1/h}&0\\ 0&e^{-iI_1/h}\end{array}\bigg) E_{m_1}=A_\pm(l,1).
\end{equation}
It is straightforward to check that for each $k\le n\le j-1$ we have 
\begin{equation*}
g(j,n+1)\cdot 2\cos(I_n/h)\cdot g(n,k)=g(j,k),
\end{equation*}
and that as a result thereof
\begin{equation}\label{induction1}
A_\pm(j,n+1)A_\pm(n,k)=A_\pm(j,k),\quad k\le n\le j-1.
\end{equation}
Now, \eqref{Tformula} is clearly true when $l=1$, since $E_{m_1}$ is necessarily of type $1^\circ$ or $4^\circ$ then, i.e., $m_1=1$ or $m_1=4$.
When all the $E_{m_j}$ are of type $1^\circ$, it is easy to see that \eqref{Tformula} holds with $A_+(l,1)$ on the right by using induction with respect to $l$ and applying \eqref{induction1} with $j=l$, $n=l-1$ and $k=1$.
If all the $E_{m_j}$ are of type $4^\circ$ one obtains \eqref{Tformula} with $A_-(l,1)$ on the right in the same way.

It remains to prove \eqref{Tformula} when $\{E_{m_1},\ldots, E_{m_l}\}$ contains at least one pair of matrices of type $2^\circ$ and $3^\circ$. Write
\begin{equation*}
F_j=\bigg(\begin{array}{cc}e^{iI_j/h}&0\\ 0&e^{-iI_j/h}\end{array}\bigg) E_{m_j},
\end{equation*}
and say that $F_j$ is of type $m^\circ$ if $m_j=m$. Since the trace is invariant under cyclic permutations, we may without loss of generality assume that $F_1$ is of type $2^\circ$. To the left of $F_1$ there must be a block of type $4^\circ$ matrices of length $k-1\ge0$, followed by a type $3^\circ$ matrix. If $k\ge2$ then the first paragraph shows that this block $F_k\cdots F_2$ of type $4^\circ$ matrices is equal to $A_-(k,2)$. It is then easy to see that
\begin{equation*}
F_{k+1} F_k\cdots F_2 F_1=F_{k+1}A_-(k,2)F_1=A_+(k+1,1).
\end{equation*} 
Now, to the left of this block there can be a block of type $1^\circ$ matrices of length $\ge0$, followed by another block of the same kind as $F_{k+1}\cdots F_1$ of length $\ge0$, and this is repeated a finite number of times until the left-hand side of \eqref{Tformula} is exhausted. But since both types of blocks have already been treated, \eqref{Tformula} follows by virtue of \eqref{induction1}.
\end{proof}

\begin{proof}[End of Proof of Theorem \ref{thm:B}]
Let $T$ be given by \eqref{Tagain1}.
By Proposition \ref{prop:persol}, $\lambda$ is an eigenvalue of $P(h)$ if and only if $\tr(T)=2$.
For $\lambda\in B_\ve(\lambda_0)$, we thus obtain \eqref{eq:theoremB} by applying Proposition \ref{prop:trace}.
In particular, \eqref{positiverealpart} gives
\begin{equation*}
\prod_{j=1}^{l}\cos(I_j/h)=2^{1-l}\exp\bigg(-\sum_{j=1}^lS_j/h\bigg)+O(h)=O(h).
\end{equation*}
Hence, for some $1\le j\le l$ we must have $\cos(I_j/h)=O(h^{1/l})$, so
\begin{equation*}
I_j/h=(\tfrac{1}{2}+k)\pi+O(h^{1/l})
\end{equation*}
which yields the quantization condition \eqref{quantizationcondition2} and the proof is complete.
\end{proof}

\section{Structure results for transition matrices}\label{section:sheets}

Here we prove Theorem \ref{thm:transitionALL} by studying each matrix of transition $1^\circ$ -- $4^\circ$ separately, starting with transition $2^\circ$. 
In view of Lemma \ref{Iaction} it suffices to consider the auxiliary transition matrices $\widetilde T_j$.

\begin{theorem}\label{thm:transition2}
Let $\widetilde T_j$ be the transition matrix defined by \eqref{tildeT_j}.
If $\widetilde T_j$ is of type $2^\circ$ then
\begin{equation*}
\widetilde T_j(\lambda,h)=e^{S_j(\mu)/h}\bigg(\begin{array}{cc} a_j(\lambda,h) & ib_j(\lambda,h) \\ ic_j(\lambda,h) & -d_j(\lambda,h)\end{array}\bigg),
\end{equation*}
where $a_j,b_j,c_j,d_j$ depend analytically on $\lambda$ in $B_\ve(\lambda_0)$ and are there equal to $1+O(h)$ uniformly as $h\to0$, and where $S_j(\mu)$ is the action integral given by \eqref{actionS}.
\end{theorem}

Before providing the full details of the proof, which requires some preparation, we briefly sketch the main idea: Write $T_j=(t_{mn})$. Taking Wronskians in \eqref{tildeT_j} we get
\begin{equation}\label{t_ij}
\begin{aligned}
t_{11}&=\frac{\mathcal W(u_{j-1}^+,\widetilde u_j^-)}{W(\widetilde u_j^+,\widetilde u_j^-)},\quad & \quad t_{12}&=\frac{\mathcal W(u_{j-1}^-,\widetilde u_j^-)}{W(\widetilde u_j^+,\widetilde u_j^-)},\\
t_{21}&=\frac{\mathcal W(\widetilde u_j^+,u_{j-1}^+)}{W(\widetilde u_j^+,\widetilde u_j^-)}, \quad & \quad t_{22}&=\frac{\mathcal W(\widetilde u_j^+,u_{j-1}^-)}{W(\widetilde u_j^+,\widetilde u_j^-)}.
\end{aligned}
\end{equation}
Here, $\mathcal W(u_{j-1}^+,\widetilde u_j^-)$ and $\mathcal W(\widetilde u_j^+,\widetilde u_j^-)$ can easily be computed using \eqref{Wronskian1+} and then estimated using Remark \ref{h-asymptotic}. However, in contrast to the case studied in Theorem \ref{thm:A}, we will not be able to use the Wronskian formula \eqref{Wronskian2} for solutions of the same type to handle $\mathcal W(\widetilde u_j^+,u_{j-1}^+)$ and $\mathcal W(u_{j-1}^-,\widetilde u_j^-)$, for these will no longer exhibit the same rapid decay. Instead, we shall express one of the WKB solutions in each Wronskian in the coordinates of a different sheet of the Riemann surface, thereby changing the type from $u^\pm$ to $u^\mp$. We can then use \eqref{Wronskian1+} to compute the Wronskian as normal, and estimate the result using Remark \ref{h-asymptotic}. For the computation of $\mathcal W(\widetilde u_j^+,u_{j-1}^-)$, the presence of branch cuts means that although the WKB solutions already are of different type, similar techniques have to be used to ensure that Remark \ref{h-asymptotic} is applicable.

The following observations are stated in sufficient generality to be useful in the sequel, but to anchor the discussion we use the assumptions of Theorem \ref{thm:transition2} as starting point, with primary goal of rewriting $\widetilde u_j^+(x)=u^+(x;{x_{2j}},y_j)$ as a solution of type $u^-$ in order to allow for the computation of $\mathcal W(\widetilde u_j^+,u_{j-1}^+)$.
Let $\mathcal R(x_0,\theta)$ denote the operator acting through rotation around $x_0$ by $\theta$ radians, so that, e.g., $\mathcal R(0,\theta)x=e^{i\theta}x$. Since $V-\mu$ is analytic it follows that if $V(x_0)-\mu=0$ then
\begin{equation}\label{eq:rotationformula}
V(\mathcal R(x_0,2\pi k)t)-\mu=e^{2i\pi k}(V(t)-\mu),\quad k\in\Z,
\end{equation}
i.e., when $t$ is rotated $2\pi k$ radians anticlockwise around $x_0$ then $V(t)-\mu$ is rotated $2\pi k$ radians anticlockwise around the origin. (Negative $k$ results in clockwise rotation by $2\pi|k|$ radians.)
We of course have similar behavior for $V+\mu$ when $V(x_0)+\mu=0$.

\begin{definition}\label{def:sheets}
Let ${x_{2j}}$ be a turning point such that $V({x_{2j}})+\mu=0$ (transitions $2^\circ$ and $4^\circ$). The point over $y_j$ that is obtained when rotating $y_j$ clockwise once around ${x_{2j}}$ will be denoted by $\hat y_j$, i.e.,
\begin{equation*}
\hat y_j=\mathcal R({x_{2j}},-2\pi)y_j.
\end{equation*}
More generally, the sheet reached (from the usual sheet) by entering the cut starting at ${x_{2j}}$ {\it from the left} will be referred to as the $\hat x$-sheet. 
The point over $y_j$ that is obtained when rotating $y_j$ anticlockwise once around ${x_{2j}}$ will be denoted by $\check y_j$, i.e.,
\begin{equation*}
\check y_j=\mathcal R({x_{2j}},2\pi)y_j.
\end{equation*}
The sheet reached (from the usual sheet) by entering the cut starting at ${x_{2j}}$ {\it from the right} will be referred to as the $\check x$-sheet. If instead $V({x_{2j}})-\mu=0$ (transitions $1^\circ$ and $3^\circ$) then all directions are to be reversed.
\end{definition}

When winding this way around a turning point we always assume that the path is appropriately deformed so as not to be obstructed by other branch cuts.
Informally, we think of $\hat x$ as lying in the sheet ``above'' the usual sheet, 
and $\check x$ as lying in the sheet ``below'' the usual sheet. The next lemma describes the relative direction of the branch cut starting at ${x_{2j-1}}$.

\begin{lemma}\label{lemma:sheet}
If $V({x_{2j-1}})-\mu=0$ then the $\hat x$-sheet is reached (from the usual sheet) by rotating anticlockwise once around ${x_{2j-1}}$, i.e., by entering the cut from the left. The $\check x$-sheet is reached (from the usual sheet) by rotating clockwise once around ${x_{2j-1}}$. If instead $V({x_{2j-1}})+\mu=0$ then the directions are reversed.
Moreover,
\begin{align}
H(z(\hat x))&=-iH(z(x)), \label{eq:Hhat}\\
H(z(\check x))&=iH(z(x)). \label{eq:Hcheck}
\end{align}
\end{lemma}

The reason for wanting to reverse the directions in Definition \ref{def:sheets} when $V({x_{2j}})-\mu=0$ is to make sure that \eqref{eq:Hhat}--\eqref{eq:Hcheck} are always in force. In fact, these identities can be taken as definitions of the sheets.

\begin{proof}
Assume first that $V({x_{2j}})+\mu=0$.
Fix a point $x$ on the line segment from ${x_{2j-1}}$ to ${x_{2j}}$ in the area between the cuts, then $\hat x=\mathcal R({x_{2j}},-2\pi)x$.
When $x$ is rotated $-2\pi$ radians around ${x_{2j}}$, $V(x)+\mu$ is rotated $-2\pi$ radians around the origin, so
\begin{equation*}
H(z(\hat x))=\bigg(\frac{e^{-2\pi i}(V(x)+\mu)}{V(x)-\mu}\bigg)^{1/4}
=-iH(z(x)).
\end{equation*}
If instead $V({x_{2j}})-\mu=0$, then $\hat x=\mathcal R({x_{2j}},2\pi)x$, which by \eqref{eq:rotationformula} again yields
\begin{equation*}
H(z(\hat x))=\bigg(\frac{V(x)+\mu}{e^{2\pi i}(V(x)-\mu)}\bigg)^{1/4}
=-iH(z(x)).
\end{equation*}
Hence, \eqref{eq:Hhat} is still in force. One proves \eqref{eq:Hcheck} in the same way.

Next, suppose that $V({x_{2j-1}})-\mu=0$. To see that the $\hat x$-sheet is reached by entering the cut starting at ${x_{2j-1}}$ from the left, take $x$ on the line segment from ${x_{2j-1}}$ to ${x_{2j}}$ and note that
\begin{equation*}
H(z(\mathcal R({x_{2j-1}},2\pi)x))=\bigg(\frac{V(x)+\mu}{e^{2\pi i}(V(x)-\mu)}\bigg)^{1/4}
=-iH(z(x)),
\end{equation*}
thus $H(z(\hat x))=H(z(\mathcal R({x_{2j-1}},2\pi)x))$, which means that $\hat x=\mathcal R({x_{2j-1}},2\pi)x$. (For a fixed point $x$, $H(z({\hdot}))$ takes distinct values at each of the four points over $x$.) In the same way one proves the statement concerning the $\check x$-sheet, as well as the reverse statements when $V({x_{2j-1}})+\mu=0$. We omit the details.
\end{proof}

In order to compute $\mathcal W(\widetilde u_j^+,u_{j-1}^+)$ we want to express $\widetilde u_j^+$ in the coordinates of a different sheet over $y_j$, and continue the resulting function through the branch cut into the domain of $u_{j-1}^+$ in the usual sheet (containing the amplitude base point $y_{j-1}$), so that their domains intersect and \eqref{Wronskian1+} and Remark \ref{h-asymptotic} are applicable. 
Under the assumptions of Theorem \ref{thm:transition2} (transition $2^\circ$),
entering the branch cut starting at ${x_{2j}}$ from the left leads to the $\hat x$-sheet by Definition \ref{def:sheets}, so this means rewriting $\widetilde u_j^+$ in the coordinates of the $\hat x$-sheet. (The same is true for transition $4^\circ$; for transitions $1^\circ$ and $3^\circ$ it means rewriting $\widetilde u_j^+$ in the coordinates of the $\check x$-sheet.) Recall that $\widetilde u_j^+(x)=u^+(x;{x_{2j}},y_j)$, which by definition means that for $x$ near $y_j$ in the usual sheet we have
\begin{equation}\label{u_3plustype}
\widetilde u_j^+(x)
=e^{z(x;{x_{2j}})/h}\bigg( \begin{array}{cc} 1& 1\\ -1& 1\end{array} \bigg)Q(z(x))\bigg( \begin{array}{cc} 0 & 1\\ 1 & 0\end{array} \bigg)
w^+(x,h;y_j),
\end{equation}
where
\begin{equation*}
z(x)=z(x;{x_{2j}})=i\int_{x_{2j}}^x\sqrt{V(t)^2-\mu^2}\de t.
\end{equation*}
The change of variables $s=\mathcal R({x_{2j}},-2\pi)t$ gives
\begin{equation}\label{zchange}
z(x;{x_{2j}})=i\int_{x_{2j}}^{\hat x}\sqrt{e^{2i\pi}(V(s)^2-\mu^2)}\de s=-z(\hat x;{x_{2j}}).
\end{equation}
Using the shorthand $z=z(x)$, $\hat z=z(\hat x)$,
we have $H(z)=iH(\hat z)$ by \eqref{eq:Hhat}, so the definition of $Q(z)$ (see \eqref{defQ}) gives
\begin{equation}\label{Qchange}
Q(z)=-i\bigg( \begin{array}{cc} H(\hat z)^{-1} & H(\hat z)^{-1}\\ -iH(\hat z) & iH(\hat z)\end{array} \bigg)=-iQ(\hat z)\bigg( \begin{array}{cc} 0 & 1\\ 1 & 0\end{array} \bigg).
\end{equation}
Note that squaring the rightmost matrix gives the identity matrix.

Next, we express $w^+(x,h;y_j)$ in the coordinates of the $\hat x$-sheet. To this end, let $w_n^\pm(z)$, $n\ge1$, be a family of solutions to the transport equations \eqref{transportodd}--\eqref{transporteven}, satisfying the initial conditions $w_n^\pm(z_0)=0$ at $z_0=z(y_0)$. Set $f_n^\pm(\hat z)=w_n^\pm(z)$, $n\ge1$. Inspecting \eqref{transportodd}--\eqref{transporteven} and using $d/dz=-d/d\hat z$, it is easy to check that 
\begin{equation*}
\bigg(\frac{d}{d\hat z}\mp\frac{2}{h}\bigg)f^\pm_{2n+1}(\hat z)=\frac{dH(\hat z)/d\hat z}{H(\hat z)}f^\pm_{2n}(\hat z),
\end{equation*}
\begin{equation*}
\frac{d}{d\hat z}f^\pm_{2n+2}(\hat z)=\frac{dH(\hat z)/d\hat z}{H(\hat z)}f^\pm_{2n+1}(\hat z)
\end{equation*}
with $f_n^\pm(\hat z_0)=0$ at $\hat z_0=z(\hat y_0)$.
By uniqueness of solutions it follows that $w^\pm_n(z)=w^\mp_n(\hat z)$ with $w_n^\mp(z(\hat y_0))=0$, which, in view of \eqref{amplitude}, implies that
\begin{equation}\label{amplitudechange}
w^\pm(x,h;y_0)
=\sum_{n=0}^\infty 
\bigg( \begin{array}{c} w^\pm_{2n}(z;z_0)\\ w^\pm_{2n+1}(z;z_0)\end{array} \bigg)
=\sum_{n=0}^\infty 
\bigg( \begin{array}{c} w^\mp_{2n}(\hat z;\hat z_0)\\ w^\mp_{2n+1}(\hat z;\hat z_0)\end{array} \bigg)
=w^\mp(\hat x,h;\hat y_0).
\end{equation}
In particular, $w^+(x,h;y_j)=w^-(\hat x,h;\hat y_j)$, which together with  \eqref{u_3plustype}--\eqref{Qchange} gives
\begin{equation*}
\widetilde u_j^+(x)=-ie^{-z(\hat x;{x_{2j}})/h}\bigg( \begin{array}{cc} 1& 1\\ -1& 1\end{array} \bigg)Q(z(\hat x))
w^-(\hat x,h;\hat y_j)=-iu^-(\hat x;{x_{2j}},\hat y_j)
\end{equation*}
for $x$ near $y_j$. 
This proves the first part of the following lemma.

\begin{lemma}\label{rewritinglemma}
For each transition $1^\circ$ -- $4^\circ$, let $\hat x$ and $\check x$ be defined in accordance with Definition \ref{def:sheets}.
Then
\begin{align*}
u^\pm(x;{x_{2j}},y_j)&=-iu^\mp(\hat x;{x_{2j}},\hat y_j)=iu^\mp(\check x;{x_{2j}},\check y_j),\\
u^\mp(x;{x_{2j-1}},\bar y_j)&=-iu^\pm(\hat x;{x_{2j-1}},\hat{\bar y}_j)=iu^\pm(\check x;{x_{2j-1}},\check{\bar y}_j).
\end{align*}
\end{lemma}

\begin{proof}
The arguments above show that $u^\pm(x;{x_{2j}},y_j)=-iu^\mp(\hat x;{x_{2j}},\hat y_j)$. Since $z(\hat x)=z(\check x)$, the identity $u^\pm(x;{x_{2j}},y_j)=iu^\mp(\check x;{x_{2j}},\check y_j)$ follows by using the same arguments except with \eqref{eq:Hcheck} used in place of \eqref{eq:Hhat}. It remains to prove the identities for the turning point ${x_{2j-1}}$.
But since $z(x;{x_{2j-1}})=-z(\hat x;{x_{2j-1}})$ (compare with \eqref{zchange}) and
\begin{equation*}
u^\mp(x;{x_{2j-1}},\bar y_j)=e^{\mp z(x;{x_{2j-1}})/h}\bigg( \begin{array}{cc} 1& 1\\ -1& 1\end{array} \bigg)Q(z(x))
\bigg( \begin{array}{cc} 0 & 1\\ 1 & 0\end{array} \bigg)^{(1\mp1)/2}w^\mp(x,h;\bar y_j),
\end{equation*}
the first identity is a consequence of \eqref{Qchange} and \eqref{amplitudechange}, while the second again follows by similar arguments except with \eqref{eq:Hcheck} used in place of \eqref{eq:Hhat}.
\end{proof}

It will be convenient to record the following observation.

\begin{lemma}\label{lemma:S}
Let $S_j(\mu)$ be given by \eqref{actionS}. Then
\begin{equation}\label{eq:S}
S_j(\mu)=i\int_{x_{2j-1}(\mu)}^{x_{2j}(\mu)}(V(t)^2-\mu^2)^{1/2}\de t.
\end{equation}
\end{lemma}

\begin{proof}
It suffices to prove \eqref{eq:S} for $\mu=\mu_0$, for then the general result follows by continuity. 
To evaluate the right-hand side of \eqref{eq:S} when $\mu_0$ is real, note that the line segment from $x_{2j-1}(\mu_0)$ to $x_{2j}(\mu_0)$ is an interval on the real line. Next, recall that we have chosen a determination of $(V^2-\mu_0^2)^{1/2}$
so that it is positive at the origin, or in fact at any point $x\in\R$ in the same sheet as the origin where $V(x)^2-\mu_0^2>0$. 
In particular, $\Arg(V(s)^2-\mu_0^2)^{1/2}=0$ for real $s<x_{2j-1}$ close to $x_{2j-1}$. For real $t>x_{2j-1}$ close to $x_{2j-1}$ we write
\begin{equation}\label{eq:arg}
(V(s)^2-\mu_0^2)^{1/2}=(V(\mathcal R(x_{2j-1},\pi)t)^2-\mu_0^2)^{1/2}=(e^{i\pi}(V(t)^2-\mu_0^2))^{1/2}
\end{equation}
with $s<x_{2j-1}$ as above (rotation in the opposite direction is incorrect since it means passing through a branch cut).
Hence, the right-hand side of \eqref{eq:S} equals
\begin{equation*}
i\int_{x_{2j-1}}^{x_{2j}}(e^{-i\pi}e^{i\pi}(V(t)^2-\mu_0^2))^{1/2}\de t=
\int_{x_{2j-1}}^{x_{2j}}(e^{i\pi}(V(t)^2-\mu_0^2))^{1/2}\de t
\end{equation*}
where the integrand in the rightmost integral is non-negative by \eqref{eq:arg}. The lemma now follows by inspecting the definition of $S_j(\mu_0)$.
\end{proof}

\begin{proof}[Proof of Theorem \ref{thm:transition2}]
Recalling \eqref{t_ij}, we first compute $\mathcal W(\widetilde u_j^+,\widetilde u_j^-)$. 
According to the behavior of $\re z(x)$, 
there is a curve $\Gamma(y_j,\bar y_j)$ from $y_j$ to $\bar y_j$ along which $\re z(x)$ is strictly increasing. 
By evaluating the Wronskian at $\bar y_j$ (see \eqref{Wronskian1+}) 
we get 
\begin{equation}\label{W(u_3,u_4)}
\mathcal W(\widetilde u_j^+,\widetilde u_j^-)=4iw^+_\mathrm{even}(\bar y_j,h;y_j),
\end{equation}
where $w^+_\mathrm{even}(\bar y_j,h;y_j)=1+O(h)$ by Remark \ref{h-asymptotic}.

Now consider $\mathcal W(u_{j-1}^+,\widetilde u_j^-)$ and note that the phase base points of $u_{j-1}^+$ and $\widetilde u_j^-$ differ. We therefore rewrite $u_{j-1}^+$ as 
\begin{align}\notag
u_{j-1}^+(x)&=\exp\bigg(i\int_{x_{2j-1}}^{x_{2j}}\sqrt{V(t)^2-\mu^2}\de t/h\bigg)u^+(x;{x_{2j}},y_{j-1})
\\ & =e^{S_j/h}u^+(x;{x_{2j}},y_{j-1}),
\label{rewriteu_1}
\end{align}
where the last identity follows by virtue of Lemma \ref{lemma:S}. 
Since we can find a curve $\Gamma(y_{j-1},\bar y_j)$ along which $\re z(x)$ is strictly increasing we can evaluate the Wronskian at $\bar y_j$ (see \eqref{Wronskian1+}) and get
\begin{equation*}
\mathcal W(u_{j-1}^+,\widetilde u_j^-)=4ie^{S_j/h}w^+_\mathrm{even}(\bar y_j,h;y_{j-1}),
\end{equation*}
where $w^+_\mathrm{even}(\bar y_j,h;y_{j-1})=1+O(h)$ by Remark \ref{h-asymptotic}. Thus, by  \eqref{t_ij} and \eqref{W(u_3,u_4)} we have
\begin{equation*}
t_{11}=e^{S_j/h}a_j,\quad a_j=\frac{w^+_\mathrm{even}(\bar y_j,h;y_{j-1})}{w^+_\mathrm{even}(\bar y_j,h;y_j)}=1+O(h),\quad h\to0.
\end{equation*}
It is clear that $a_j$ depends analytically on $\lambda$ since this is true for the amplitude functions $w^+_\mathrm{even}$.

We now compute $\mathcal W(\widetilde u_j^+,u_{j-1}^+)$. 
By Lemma \ref{rewritinglemma} we have $\widetilde u_j^+(x)=-iu^-(\hat x;{x_{2j}},\hat y_j)$ for $\hat x$ near $\hat y_j$. 
Take the function on the right and continue it through the branch cut starting at ${x_{2j}}$ into the domain in the usual sheet containing $y_{j-1}$. We remark that at $y_{j-1}$ it takes the value $-iu^-(y_{j-1};{x_{2j}},\hat y_j)$. Similarly, writing $u_{j-1}^+(x)=e^{S_j/h}u^+(x;{x_{2j}},y_{j-1})$ as before, we see that $u_{j-1}^+(\hat y_j)=e^{S_j/h}u^+(\hat y_j;{x_{2j}},y_{j-1})$, so by evaluating the Wronskian at $\hat y_j$ (see \eqref{Wronskian1+}) we get
\begin{equation*}
\mathcal W(\widetilde u_j^+,u_{j-1}^+)=
-\mathcal W(u_{j-1}^+,\widetilde u_j^+)=-4e^{S_j/h}w^+_\mathrm{even}(\hat y_j,h;y_{j-1}).
\end{equation*}
Since $\re z(\hat x)$ is a strictly increasing function of $\im \hat x$ near $\hat x=\hat y_j$, we can find a curve from $y_{j-1}$ to $\hat y_j$, passing through the branch cut at ${x_{2j}}$, along which $t\mapsto \re z(t)$ is strictly increasing (compare with Figure \ref{figure2}). 
Hence, by \eqref{t_ij}, \eqref{W(u_3,u_4)} and Remark \ref{h-asymptotic} we get 
\begin{equation*}
t_{21}=
ie^{S_j/h}c_j,\quad c_j=\frac{w^+_\mathrm{even}(\hat y_j,h;y_{j-1})}{w^+_\mathrm{even}(\bar y_j,h;y_j)}=1+O(h),\quad h\to0.
\end{equation*}

Let us consider $\mathcal W(u_{j-1}^-,\widetilde u_j^-)$ next. 
We fix the domain of $u_{j-1}^-$ and express $\widetilde u_j^-$ in the coordinates of the sheet reached when passing through the branch cut at ${x_{2j-1}}$ from the left. For transition $2^\circ$ this is the $\hat x$-sheet according to Definition \ref{def:sheets}.
First note that
\begin{equation*}
\widetilde u_j^-(x)=\exp\bigg(i\int_{x_{2j-1}}^{x_{2j}}\sqrt{V(t)^2-\mu^2}\de t/h\bigg)u^-(x;{x_{2j-1}},\bar y_j)=e^{S_j/h}u^-(x;{x_{2j-1}},\bar y_j),
\end{equation*}
compare with \eqref{rewriteu_1}. Applying Lemma \ref{rewritinglemma} we get
\begin{equation*}
\widetilde u_j^-(x)=-ie^{S_j/h}u^+(\hat x;{x_{2j-1}},\hat{\bar y}_j).
\end{equation*}
We continue the expression on the right through the branch cut at ${x_{2j-1}}$, and note that at $\bar y_{j-1}$, it takes the value $-ie^{S_j/h}u^+(\bar y_{j-1};{x_{2j-1}},\hat{\bar y}_j)$. As above we can find a curve from $\hat{\bar y}_j$ to $\bar y_{j-1}$, passing through the branch cut at ${x_{2j-1}}$, along which $t\mapsto\re z(t)$ is strictly increasing. Hence, by evaluating the Wronskian at $\bar y_{j-1}$ (see \eqref{Wronskian1+}) we obtain
\begin{equation*}
\mathcal W(u_{j-1}^-,\widetilde u_j^-)=-4e^{S_j/h}w^+_\mathrm{even}(\bar y_{j-1},h;\hat{\bar y}_j).
\end{equation*}
In view of \eqref{t_ij}, \eqref{W(u_3,u_4)} and Remark \ref{h-asymptotic} we conclude that 
\begin{equation*}
t_{12}=ie^{S_j/h}b_j,\quad b_j=\frac{w^+_\mathrm{even}(\bar y_{j-1},h;\hat{\bar y}_j)}{w^+_\mathrm{even}(\bar y_j,h;y_j)}=1+O(h),\quad h\to0.
\end{equation*}

Finally, let us consider $\mathcal W(\widetilde u_j^+,u_{j-1}^-)$. To get an asymptotic estimate we will need to connect the amplitude base points of $\widetilde u_j^+$ and $u_{j-1}^-$ by a curve passing through both the branch cut at ${x_{2j-1}}$ {\it and} the branch cut at ${x_{2j}}$. In view of Definition \ref{def:sheets} this will be possible if we (in the present case of transition $2^\circ$) express $\widetilde u_j^+$ in the coordinates obtained by rotating anticlockwise {\it twice} around ${x_{2j}}$. Let the coordinates thus obtained be denoted $\hat{\hat x}$. Applying Lemma \ref{rewritinglemma} two times we get
\begin{equation*}
\widetilde u_j^+(x)=-u^+(\hat{\hat x};{x_{2j}},\hat{\hat y}_j)=-e^{(z(\hat{\hat x};{x_{2j}})-z(\hat{\hat x};{x_{2j-1}}))/h}u^+(\hat{\hat x};{x_{2j-1}},\hat{\hat y}_j)
\end{equation*}
for $x$ near $y_j$. 
We can find a path from $\hat{\hat y}_j$ to $\bar y_{j-1}$ along which $t\mapsto \re z(t)$ is strictly increasing (compare with Figure \ref{figure2}), so evaluating the Wronskian at $\bar y_{j-1}$ we get (see \eqref{Wronskian1+})
\begin{equation*}
\mathcal W(\widetilde u_j^+,u_{j-1}^-)=-4ie^{(z(\bar y_{j-1};{x_{2j}})-z(\bar y_{j-1};{x_{2j-1}}))/h}w^+_\mathrm{even}(\bar y_{j-1},h;\hat{\hat y}_j),
\end{equation*}
where $w^+_\mathrm{even}(\bar y_{j-1},h;\hat{\hat y}_j)=1+O(h)$ as $h\to0$ by Remark \ref{h-asymptotic}.
Here, 
\begin{equation}\label{eq:homotopy}
z(\bar y_{j-1};{x_{2j}})-z(\bar y_{j-1};{x_{2j-1}})= i\bigg(\int_{x_{2j}}^{\bar y_{j-1}}+\int_{\bar y_{j-1}}^{x_{2j-1}}\bigg)\sqrt{V(t)^2-\mu^2}\de t,
\end{equation}
where the path of integration is homotopic to a curve starting at ${x_{2j}}$ and following a curve in the $\hat x$-sheet through the branch cut at ${x_{2j-1}}$, then arriving at ${x_{2j-1}}$ via $\bar y_{j-1}$. In other words, it is homotopic to the path from ${x_{2j}}$ to ${x_{2j-1}}$ in the $\hat x$-sheet. 
If we rotate back to the usual sheet using \eqref{eq:rotationformula} and then reverse the integration direction, 
we find that
\begin{align*}
z(\bar y_{j-1};{x_{2j}})-z(\bar y_{j-1};{x_{2j-1}})&= i\int_{x_{2j}}^{x_{2j-1}}\sqrt{e^{2i\pi}(V(t)^2-\mu^2)}\de t
\\ & =i\int_{x_{2j-1}}^{x_{2j}}\sqrt{V(t)^2-\mu^2}\de t
\end{align*}
which by Lemma \ref{lemma:S} is equal to the action integral $S_j$ given by \eqref{actionS}. 
Recalling \eqref{t_ij}, \eqref{W(u_3,u_4)} we get
\begin{equation*}
t_{22}=-e^{S_j/h}d_j,\quad d_j=\frac{w^+_\mathrm{even}(\bar y_{j-1},h;\hat{\hat{y}}_{j-1})}{w^+_\mathrm{even}(\bar y_j,h;y_j)}=1+O(h),\quad h\to0,
\end{equation*}
which completes the proof.
\end{proof}

We now turn to the matrices of transition $1^\circ$, $3^\circ$ and $4^\circ$.

\begin{theorem}\label{thm:transition3}
Let $\widetilde T_j$ be the transition matrix defined by \eqref{tildeT_j}.
If $\widetilde T_j$ is of type $3^\circ$ then
\begin{equation*}
\widetilde T_j(\lambda,h)=e^{S_j(\mu)/h}
\bigg(\begin{array}{cc} a_j & -i b_j \\ -ic_j & -d_j\end{array}\bigg),
\end{equation*}
where $a_j,b_j,c_j,d_j$ depend analytically on $\lambda$ in $B_\ve(\lambda_0)$ and are there equal to $1+O(h)$ uniformly as $h\to0$, and where $S_j(\mu)$ is the action integral given by \eqref{actionS}.
\end{theorem}

\begin{proof}
Write $\widetilde T_j=(t_{mn})$, then $t_{mn}$ is given by \eqref{t_ij}. The proof continues in almost identical fashion to the proof of Theorem \ref{thm:transition2}. The only difference is that when rewriting $\widetilde u_j^+$ and $\widetilde u_j^-$ in order to compute the Wronskians $\mathcal W(\widetilde u_j^+,u_{j-1}^+)$ and $\mathcal W(u_{j-1}^-,\widetilde u_j^-)$, respectively, we do so using the coordinates of the $\check x$-sheet instead of the $\hat x$-sheet, see the discussion preceding \eqref{u_3plustype}. In view of Lemma \ref{rewritinglemma}, this results in the loss of a factor $-1$ in the formulas for $t_{21}$ and $t_{12}$. To compute $\mathcal W(\widetilde u_j^+,u_{j-1}^-)$ we apply Lemma \ref{rewritinglemma} twice, and since $\check{\check x}=\hat{\hat x}$, this leads to the same calculations as in the proof of Theorem \ref{thm:transition2}.
\end{proof}

\begin{theorem}\label{thm:transition1}
Let $\widetilde T_j$ be the transition matrix defined by \eqref{tildeT_j}.
If $\widetilde T_j$ is of type $1^\circ$ then
\begin{equation*}
\widetilde T_j(\lambda,h)=e^{S_j(\mu)/h}\bigg(\begin{array}{cc} a_j & i b_j \\ -ic_j & d_j\end{array}\bigg),
\end{equation*}
where $a_j,b_j,c_j,d_j$ depend analytically on $\lambda$ in $B_\ve(\lambda_0)$ and are there equal to $1+O(h)$ uniformly as $h\to0$, and where $S_j(\mu)$ is the action integral given by \eqref{actionS}.
\end{theorem}

\begin{proof}
Write $\widetilde T_j=(t_{mn})$, then $t_{mn}$ is given by \eqref{t_ij}. The computation of $t_{11}$ is the same as for transitions $2^\circ$ and $3^\circ$. Since the nature of ${x_{2j-1}}$ is the same as in transition $2^\circ$, the computation of $t_{12}$ (rewriting $\widetilde u_j^-$ by rotating around ${x_{2j-1}}$) is the same as for transition $2^\circ$. Since the nature of ${x_{2j}}$ is the same as in transition $3^\circ$, the computation of $t_{21}$ (rewriting $\widetilde u_j^+$ by rotating around ${x_{2j}}$) is the same as for transition $3^\circ$. 
For $t_{22}$, we simply compute $\mathcal W(\widetilde u_j^+,u_{j-1}^-)$ directly from \eqref{Wronskian1+} by evaluating the Wronskian at $\bar y_{j-1}$ (the amplitude base point of $u_{j-1}^-$) and obtain
\begin{equation*}
\mathcal W(\widetilde u_j^+,u_{j-1}^-)=4ie^{(z(\bar y_{j-1};{x_{2j}})-z(\bar y_{j-1};{x_{2j-1}}))/h}w_\mathrm{even}^+(\bar y_{j-1},h;y_j).
\end{equation*}
In fact, for transition $1^\circ$ there is, in view of Definition \ref{def:sheets}, a path $\Gamma(y_j,\bar y_{j-1})$ from $y_j$ to $\bar y_{j-1}$ along which $\re z(x)$ is strictly increasing: starting at $y_j$ in the usual sheet, $\Gamma(y_j,\bar y_{j-1})$ goes through  the branch cut at ${x_{2j}}$ from the right (thereby entering the $\hat x$-sheet), then goes through the branch cut at ${x_{2j-1}}$ from the right, reentering the usual sheet and then ending at $\bar y_{j-1}$. In view of the discussion following \eqref{eq:homotopy}, this implies that $z(\bar y_{j-1};{x_{2j}})-z(\bar y_{j-1};{x_{2j-1}})=S_j$, which gives
\begin{equation*}
t_{22}=\frac{\mathcal W(\widetilde u_j^+,u_{j-1}^-)}{\mathcal W(\widetilde u_j^+,\widetilde u_j^-)}=e^{S_j/h}\frac{w^+_\mathrm{even}(\bar y_{j-1},h;y_j)}{w^+_\mathrm{even}(\bar y_j,h;y_j)}.
\end{equation*}
The proof is complete.
\end{proof}

\begin{theorem}\label{thm:transition4}
Let $\widetilde T_j$ be the transition matrix defined by \eqref{tildeT_j}.
If $\widetilde T_j$ is of type $4^\circ$ then
\begin{equation*}
\widetilde T_j(\lambda,h)=e^{S_j(\mu)/h}\bigg(\begin{array}{cc} a_j & -i b_j \\ ic_j & d_j\end{array}\bigg),
\end{equation*}
where $a_j,b_j,c_j,d_j$ depend analytically on $\lambda$ in $B_\ve(\lambda_0)$ and are there equal to $1+O(h)$ uniformly as $h\to0$, and where $S_j(\mu)$ is the action integral given by \eqref{actionS}.
\end{theorem}

\begin{proof}
Write $\widetilde T_j=(t_{mn})$, then $t_{mn}$ is given by \eqref{t_ij}. The computations of $t_{11}$ and $t_{22}$ are the same as for transition $1^\circ$ (except that for $t_{22}$, the path from $y_j$ to $\bar y_{j-1}$ goes via the $\check x$-sheet instead), while the computations of $t_{12}$ and $t_{21}$ are reversed: 
Since the nature of ${x_{2j-1}}$ is the same as in transition $3^\circ$, the computation of $t_{12}$ (rewriting $\widetilde u_j^-$ by rotating around ${x_{2j-1}}$) is the same as for transition $3^\circ$. Since the nature of ${x_{2j}}$ is the same as in transition $2^\circ$, the computation of $t_{21}$ (rewriting $\widetilde u_j^+$ by rotating around ${x_{2j}}$) is the same as for transition $2^\circ$.
\end{proof}

\section*{Acknowledgements}

The research of Setsuro Fujii{\'e} was supported by the JSPS Kakenhi Grant No. JP15K04971.
The research of Jens Wittsten was supported by the Knut och Alice Wallenbergs Stiftelse Grant No.~2013.0347.

\bibliographystyle{amsplain}
\bibliography{references}

\end{document}